\documentclass[a4paper,10pt,notitlepage,reqno]{article}
\usepackage[english]{babel}
\usepackage{amssymb,amsthm,bm} 
\usepackage{mathtools}
\usepackage{mathrsfs}
\usepackage{enumitem}
\usepackage{authblk}
\usepackage{cite}
\usepackage{xcolor}
\usepackage{tikz}
\usepackage{stackrel}
\usepackage{geometry}
\usepackage{hyperref}
\theoremstyle{plain} 
\newtheorem{theorem}{Theorem}[section]
\newtheorem{lemma}{Lemma}[section]
\newtheorem{proposition}{Proposition}[section]

\theoremstyle{definition}

\theoremstyle{remark}
\newtheorem{remark}{Remark}[section]

\numberwithin{equation}{section}

\DeclareMathOperator{\realpart}{Re}
\renewcommand{\Re}{\realpart}

\DeclareMathOperator{\imaginarypart}{Im}
\renewcommand{\Im}{\imaginarypart}

\DeclareMathOperator{\sgn}{sgn}

\newcommand{\R}{\mathbb{R}}
\newcommand{\C}{\mathbb{C}}

\newcommand{\normeq}[1]{{\left\vert\kern-0.25ex\left\vert\kern-0.25ex\left\vert #1 
    \right\vert\kern-0.25ex\right\vert\kern-0.25ex\right\vert}}

\usepackage{color}
\usepackage[normalem]{ulem}
\definecolor{DarkGreen}{rgb}{0,0.5,0.1} % David

\newcommand\soutD{\bgroup\markoverwith
{\textcolor{DarkGreen}{\rule[.5ex]{2pt}{1pt}}}\ULon}

\newcommand{\Hm}[1]{\leavevmode{\marginpar{\tiny%
$\hbox to 0mm{\hspace*{-1.5mm}$\leftarrow$\hss}%
\vcenter{\vrule depth 0.1mm height 0.1mm width \the\marginparwidth}%
\hbox to
0mm{\hss$\rightarrow$\hspace*{-1.5mm}}$\\\relax\raggedright #1}}}

\definecolor{Darkgblue}{rgb}{0.3,0.3,0.5}

\title{\textbf{Uniform resolvent estimates 
and absence of eigenvalues of 
biharmonic operators with complex potentials}}

\author[1]{Lucrezia Cossetti} 
\author[1,2]{Luca Fanelli}		
\author[3]{David Krej\v{c}i\v{r}\'ik}

\affil[1]{Ikerbasque \& Universidad del Pa\'is Vasco/Euskal Herriko Unibertsitatea, UPV/EHU, \newline Aptdo. 644, 48080, Bilbao, Spain; lucrezia.cossetti@ehu.eus; luca.fanelli@ehu.es}

\affil[2]{BCAM-Basque Center for Applied Mathematics, Mazarredo, 14 E48009 Bilbao, Spain}

\affil[3]{Faculty of Nuclear Sciences and Physical Engineering, Czech Technical University in Prague, \newline Trojanova 13, 12000 Prague 2, Czechia; david.krejcirik@fjfi.cvut.cz}

\begin{document}

\date{\small 11 September 2023}
%%%%%%%%%%%%%%%%%%%%%%%%%%%%%%%%%%%%%%%%%%%%%%%%%%%%%%%%%%%%%%%%%%%%%%%%%%%%%%%%%%%%%%%%%%%%%%%%%%%%%%%%%%%%%%%%%%%%%%%%%%%%%%%%%%%%%%%%%%
\maketitle

%\vspace{-1cm}

%\nocite{*}

%------------abstract------------%

\begin{abstract}
	\noindent	
	We quantify the subcriticality of the bilaplacian in  dimensions greater than four by providing explicit repulsivity/smallness conditions on complex additive perturbations under which the spectrum remains stable. Our assumptions cover critical Rellich-type potentials too. As a byproduct we obtain uniform resolvent estimates in  weighted spaces. Some of the results are new also in the self-adjoint setting.
\end{abstract}

%\footnotetext{\emph{2020 Mathematics Subject Classification: 35A02, 35M10, 35Q74}. 
%}

%\footnotetext{\emph{Keywords}. surface electromigration equation, non-local non linearity, uniqueness of solution, Carleman estimates}

%---------------------------------%
\section{Introduction}
This paper is concerned with the study of spectral properties of biharmonic operators with complex potentials 
\begin{equation}\label{eq:biharmonic-op}
	H_V:=\Delta^2 + V
	\qquad \text{in} \quad L^2(\R^d)
\end{equation}
in higher dimensions $d\geq 5.$ Here $H_0:=\Delta^2$ is the forth-order differential operator known as the bilaplacian and~$V$ is the operator of multiplication by a suitable generating function 
$V\colon \R^d \to \C$. 

\medskip
It is well known that the spectrum of the self-adjoint realisation of $H_0$ is purely absolutely continuous and coincides with the non-negative semi-axis $[0,\infty).$ If $d\leq 4,$ the operator $H_0$ is \emph{critical} in the sense that $\inf \sigma(H_0 + V)<0$ whenever $V\in C^\infty_0(\R^d)$ is real-valued, non-positive and non-trivial. In other words, in the low dimensions, the introduction of any 
potential of the above type (called attractive) always creates discrete eigenvalues below the essential spectrum.  

On the other hand, in higher dimensions $d\geq 5,$ the operator is \emph{subcritical} in the sense that its spectrum remains stable under small perturbations. This property is a consequence of the existence of the celebrated Rellich inequality 
\begin{equation}\label{eq:Rellich}
	\int_{\R^d} |\Delta \psi|^2 \geq C_\textup{R} \int_{\R^d} \frac{|\psi|^2}{r^4}, \qquad \psi\in H^2(\R^d), \quad d\geq 5,
\end{equation}
with $C_\textup{R}$ being the sharp constant $C_\textup{R}:=d^2(d-4)^2/16$
and $r(x):=|x|$. 

Inequality~\eqref{eq:Rellich} can be seen as a higher-order generalisation of the classical Hardy inequality
\begin{equation}\label{eq:Hardy}
	\int_{\R^d} |\nabla \psi|^2 \geq C_\textup{H} \int_{\R^d} \frac{|\psi|^2}{r^2}, 
	\qquad \psi\in H^1(\R^d),
	\quad d\geq 3,
\end{equation}
with $C_\textup{H}$ being the sharp constant $C_\textup{H}:=(d-2)^2/4.$

\medskip
This paper arises as a natural complement of the recent results on non-self-adjoint biharmonic operators obtained in~\cite{IbrogimovKrejcirikLaptev2021}. Motivated by the \emph{criticality} of the bilaplacian in lower dimensions, in that work the authors were interested in locating in the complex plane the eigenvalues of~$H_V$ created by the introduction of the potential~$V.$ In other words, 
they extended to the biharmonic operators 
the analogous spectral enclosures
available for possibly non-self-adjoint Schrödinger operators \cite{Frank11,Frank18,FrankLaptevLiebSeiringer06,LaptevSafronov09,Safronov10,FrankSimon17,Enblom16,LeeSeo19,IbrogimovStampach19},  initiated with the celebrated paper of Davies \emph{et al.}~\cite{AbramovAslanyanDavies01} in 2001.
As a matter of fact, these results have been already generalised to lower order operators, such as Dirac and fractional Schrödinger models~\cite{BogliCuenin22,SafronovLaptevFerrulli19,Cuenin19,CassanoPizzichilloVega20,CassanoIbrogimovKrejcirikStampach20,Cuenin17,CueninLaptevTretter14,FanelliKrejcirik19,DAnconaFanelliSchiavone21} and to other second order operators~\cite{CassanoCossettiFanelli21-lame,CassanoCossettiFanelli21,Cossetti22,KrejcirikKurimaiova20}.

\medskip
In the present work, instead, motivated by the \emph{subcriticality} of the bilaplacian in $d\geq 5,$ we are interested in the complementary problem of finding physically natural conditions on the potential for guaranteeing that no eigenvalues are created. 
For discrete eigenvalues and real-valued potentials~$V$,
elementary sufficient conditions follow by the Rellich inequality~\eqref{eq:Rellich}.
In order to cover non-self-adjoint perturbations
and embedded eigenvalues, we suitably develop
the \emph{method of multipliers}. 

This technique saw its origin in a purely partial differential equations setting, being primarily introduced by Morawetz in her celebrated paper~\cite{Morawetz1968} to understand characterising properties of solutions to the nonlinear Klein-Gordon equation and then fruitfully developed for several other models (see~\cite{IkebeSaito1972,PerthameVega1999,PerthameVega2008,FanelliVega2009,Fanelli09,Zubeldia2012,Zubeldia2014,BarceloVegaZubeldia2013,BarceloFanelliRuizVilela2013,CacciafestaDAnconaLuca16,Cacciafesta11,CassanoDAncona16,BoussaidDAnconaFanelli11} for some selected literature). 

Nonetheless, in the last decades, this method has been intensively used in functional analysis, too. 
The first, physically satisfactory, application of the method of multipliers in spectral theory can be found in~\cite{FanelliKrejcirikVega18-JST}: here the authors established sufficient conditions which guarantee the \emph{total} absence of eigenvalues of electromagnetic Schrödinger operators. The remarkable feature of this work, compared to previous ones, is that the aforementioned conditions are compatible with the well established gauge invariance of electromagnetic models, moreover it covers also non-self-adjoint potentials. The robustness of the method of multipliers as a tool in spectral theory has been demonstrated in its successful application to different models: see~\cite{FanelliKrejcirikVega18-JFA,CossettiFanelliKrejcirik20,CossettiKrejcirik20} for the results on Schrödinger operators in different settings,~\cite{CossettiFanelliKrejcirik20} for Dirac equation and~\cite{Cossetti17} for the Lamé operator of elasticity. 

\subsection{The main results}
In this paper, we present an extension of the method of multipliers
to higher order differential (not necessarily self-adjoint) operators.
As a first result, we will show a total absence of eigenvalues under suitably small, complex-valued potential. More precisely, we have the following result.

\begin{theorem}[Total absence of eigenvalues]\label{thm:main-nsa}
	Let $d\geq 5.$ Suppose that $V\colon \R^d \to \C$ is such that $V\in L^1_\textup{loc}(\R^d)$ and $r^2 V\in L^2_\textup{loc}(\R^d)$.
	 %\txtD{where $r(x):=|x|$}.
	Moreover, assume that 
	\begin{equation}\label{eq:smallness}
		\forall\, \psi \in H^2(\R^d),
		\qquad \int_{\R^d} r^4 |V|^2 |\psi|^2\leq a^2 \int_{\R^d} \frac{|\nabla \psi|^2}{r^2},
	\end{equation}
	 where $a$ is such that 
	\begin{equation}\label{eq:smallness-a}
		\frac{4d^2(d-3)}{(d-2)(d-4)} \frac{a}{\sqrt{C_\textup{H}}} + \frac{4d \sqrt{d}}{(d-4)\sqrt{(d-2)(d-4)}}\frac{a^{3/2}}{\sqrt{C_\textup{H}}^{3/2}} <1,
	\end{equation} 
	with $C_\textup{H}$ being the Hardy constant (\emph{cfr.~\eqref{eq:Hardy}}). Then $\sigma_\mathrm{p}(H_V)=\varnothing.$
\end{theorem}

\begin{remark}
Notice that condition~\eqref{eq:smallness} is intrinsically a smallness condition with respect to the free Hamiltonian $\Delta^2.$ This can be seen using the Hardy-Rellich inequality
\begin{equation}\label{eq:HR}
\int_{\R^d}|\Delta \psi|^2
\geq C_\textup{HR}\int_{\R^d} \frac{|\nabla \psi|^2}{r^2},
\end{equation}
valid for any $\psi\in H^2(\R^d)$, $d\geq 5$ and with $C_\textup{HR}$ being the sharp constant $C_\textup{HR}= d^2/4$ (see~\cite{TertikasZographopoulos2007}) Indeed, if $V$ satisfies~\eqref{eq:smallness} then, in particular, it satisfies
\begin{equation*}
	\int_{\R^d} r^4 |V|^2 |\psi|^2\leq a^2 \frac{4}{d^2} \int_{\R^d} |\Delta\psi|^2.
\end{equation*}
Notice that if $a$ satisfies~\eqref{eq:smallness-a}, in particular $a^24/d^2<1.$
\end{remark}

The total absence of eigenvalues of Theorem~\ref{thm:main-nsa}
is a consequence of two independent results:
Absence of eigenvalues in a cone containing the positive real semi-axis
and absence of eigenvalues in the rest of the complex plane.
As usual, the former is more demanding for it involves
embedded eigenvalues too (traditionally referred to as ``positive''
eigenvalues in the self-adjoint setting). 
On the other hand, the latter (discrete) eigenvalues
can be excluded more easily by numerical range arguments,
which is the content of the following result.

\begin{theorem}[Absence of non-positive eigenvalues]\label{thm:absence-out}
	Let $d\geq 5.$ Suppose that $V\colon \R^d\to \C$ is such that $V\in L^1_\textup{loc}(\R^d)$ and $r^2V\in L^2_\textup{loc}(\R^d).$ 
	Moreover,  given $\delta>0$ assume that 
	\begin{equation}\label{eq:V-cond}
	\forall\, \psi \in H^2(\R^d),
	\qquad
	\int_{\R^d} r^4 |V|^2 |\psi|^2\leq a_\delta^2 \int_{\R^d} |\Delta \psi|^2,
	\end{equation}
	where $a_\delta>0$ is such that $\left(1+\frac{1}{\delta} \right)\frac{a_\delta}{\sqrt{C_\textup{R}}}<1,$ with $C_\textup{R}$ being the Rellich constant which appears in~\eqref{eq:Rellich}. 
	Then
	\begin{equation*}
	\sigma_\mathrm{p}(\Delta^2+ V)\subseteq \{z\in \C\colon |\Im z|\leq \delta \Re z\}.
	\end{equation*}
\end{theorem}

\begin{remark}

We emphasise that the critical Rellich potential $V_\alpha(x):=\frac{\alpha}{|x|^4}$ for suitably small coupling constant $\alpha\in\C$ is covered by Theorems~\ref{thm:main-nsa} and~\ref{thm:absence-out}. Indeed, one has
\begin{equation}\label{eq:Rellich-potential}
	 \int_{\R^d} |x|^4|V_\alpha|^2|\psi|^2
	= |\alpha| \int_{\R^d} \frac{|u|^2}{|x|^4}.
\end{equation} 
If $|\alpha|$ is sufficiently small, one can see that the right-hand-side of~\eqref{eq:Rellich-potential} satisfies conditions~\eqref{eq:smallness} and~\eqref{eq:V-cond}; this just follows using a weighted Hardy inequality (see~\eqref{eq:weighted-Hardy} below) and the Rellich inequality~\eqref{eq:Rellich}, respectively.  
\end{remark}

As a further application of the technique developed to prove Theorems~\ref{thm:main-nsa} and~\ref{thm:absence-out}, we establish uniform resolvent estimates for $H_V$. The following result represents a higher order analogue of the results available for Schrödinger operators (see, for instance,~\cite{CossettiKrejcirik20,BarceloVegaZubeldia2013}).

\begin{theorem}[Uniform resolvent estimates]\label{thm:resolvent-est}
	Let $d\geq 5.$ 
	Under the same hypotheses of Theorem~\ref{thm:main-nsa},
    there exists a positive constant $c$ such that, for all $z\in \C,$
	\begin{equation*}
		\|r^{-2}(H_V-z)^{-1}r^{-2}\|_{L^2(\R^d)\to L^2(\R^d)} \leq c.
	\end{equation*}
\end{theorem}

Actually, Theorem~\ref{thm:resolvent-est} is a direct consequence of a stronger result, which shows that \emph{a priori} estimates for solutions to the resolvent equation hold. 
\begin{theorem}[\emph{A priori} estimates]\label{thm:apriori}
	Let $d\geq 5.$ 
	Under the same hypotheses of Theorem~\ref{thm:main-nsa},
	there exists a positive constant $c$ such that, given any $z\in \C$ and $r^2f\in L^2(\R^d),$ any solution $u\in H^2(\R^d)$ of the equation $(H_V-z)u=f$ satisfies
	\begin{itemize}
		\item for $\Re z\geq 0,$ 
		\begin{equation}\label{eq:lr}
			\sum_{j=1}^d\|\nabla(\partial_j u)^-\|_{L^2(\R^d)}\leq c \|r^2 f\|_{L^2(\R^d)},
		\end{equation}
		\item for $\Re z<0,$
		\begin{equation}\label{eq:ud}
			\|\Delta u\|_{L^2(\R^d)}\leq c_{a,d}\|r^2 f\|_{L^2(\R^2)},
		\end{equation}
		where
		\begin{equation}\label{eq:cad}
			c_{a,d}:=\frac{\sqrt{C_\textup{HR}}}{\sqrt{C_\textup{R}}\sqrt{C_\textup{HR}}-a},
		\end{equation}
		with $C_\textup{R}$ and $C_\textup{HR}$ being the Hardy and Hardy-Rellich constants respectively (\emph{cfr.}~\eqref{eq:Rellich} and~\eqref{eq:HR}) and with $a$ as in~\eqref{eq:smallness-a}. (One can check that if $a$ satisfies~\eqref{eq:smallness-a} then $c_{a,d}$ is strictly positive and bounded).
	\end{itemize}
	Here, for any suitable function $v,$ we denoted with $v^-$ the auxiliary function defined as follows
	\begin{equation}\label{eq:auxiliary}
		v^-(x):=e^{-i(\Re \sqrt{z})^{1/2}\sgn(\Im \sqrt{z})|x|} v(x).
	\end{equation}
\end{theorem}

\begin{remark}\label{rmk:gauge}
As one can notice in the statement of Theorem~\ref{thm:apriori}, when the spectral parameter $z$ lies in a region containing the essential spectrum (here the non-negative real line), one is able to prove an \emph{a priori} estimate \emph{only} for a suitable \emph{change of gauge} of $u$ and not for the solution $u$ itself (compare~\eqref{eq:lr} with~\eqref{eq:ud}). 
This fact is not just a technicality strictly related to the biharmonic operator. As a matter of fact, the need of a change of gauge in a region close to the essential spectrum was already observed for Schrödinger operators in different contexts (see~\cite[Thm.2.1]{BPST2004},~\cite[Thm.1.6]{BarceloVegaZubeldia2013} and~\cite[Thm.8]{FanelliKrejcirikVega18-JST}) and can be easily explained looking at the toy model given by Helmholtz operators,
namely $\Delta +\kappa,$ with $\kappa\in \R_{\geq 0},$ where with $\R_{\geq 0}$ we will be denoting the set of non-negative real numbers:

\begin{quote}
Let us consider for simplicity the physical dimension $d=3.$ In this case the Green function of this operator is given by 
	$\mathcal{G}_\kappa^{\pm}(x)
	=\tfrac{1}{4\pi} \tfrac{e^{\pm i\sqrt{\kappa}|x|}}{|x|}.
	$
Computing $\nabla \mathcal{G}_\kappa^{+}$ one has
	$|\nabla \mathcal{G}_\kappa^{+}(x)|\sim |x|^{-2} + |x|^{-1},$
where the first term comes from the differentiation of $ |x|^{-1}$ and the second term comes from the differentiation of $e^{i\sqrt{\kappa}|x|}.$ If, on one hand, for any sphere $B$ in $\R^3$ the inverse square $|x|^{-2}$ belongs to $L^2(\R^3\setminus B),$ on the other hand, $|x|^{-1}$ does not display a sufficiently fast decay to ensure the $L^2$-integrability at infinity. Thus $\nabla \mathcal{G}_\kappa^{+}\notin L^2(\R^3\setminus B).$ Nevertheless if one introduces the auxiliary function 
$e^{-i\sqrt{\kappa}|x|} \mathcal{G}_\kappa^{+}$ 
one has $\nabla (e^{-i\sqrt{\kappa}|x|} \mathcal{G}_\kappa^{+})\sim |x|^{-2}\in L^2(\R^3\setminus B).$ Taking this into account, the only reasonable $L^2$-bound expected is the one for the gradient of a change of gauge of the solution and \emph{not} of the solution itself. 
This justify the reasonableness of~\eqref{eq:lr}.
\end{quote}

Another related motivation for the change of gauge has to be found in the Sommerfeld radiation conditions introduced for guaranteeing uniqueness of solution to the Helmholtz equation (see~\cite{Sommerfeld1912} or~\cite{Schot1992} for a review reference). It is well known that the mathematical description of the propagation of given acoustic, elastic or electromagnetic waves after encountering an obstacle is given through an exterior boundary value problem for the Helmholtz equation $\Delta u + \kappa u=0,$ $\kappa \in \R_{\geq 0}.$ The main difficulty that arises working with this model is the non uniqueness of the solution:
besides the expected outgoing waves which result when the incident wave is scattered by the object, the mathematical solution of the problem also provides incoming waves which originate at infinity and move towards the object. These incoming waves are physically meaningless and must be rejected by some criterion built into the mathematical formulation of the problem. Sommerfeld, in his pioneering work~\cite{Sommerfeld1912}, was the first who stated a mathematically precise and easily applicable condition for guaranteeing the uniqueness of solution $u$ to the Helmholtz equation, namely
\begin{equation}\label{eq:Sommerfeld}
	\lim_{|x|\to \infty} |x|^{\frac{d-1}{2}} \left|\left(\frac{\partial}{\partial |x|} - i \sqrt{\kappa} \right) u(x)\right|=0.
\end{equation}
Notice that, for instance, between $u_{\pm}(x)=\frac{e^{\pm i\sqrt{\kappa}|x|}}{4\pi |x|},$ which are both solutions to 
the three-dimensional Helmholtz equation with a point source at $0,$ only $u_+$ satisfies condition~\eqref{eq:Sommerfeld}. 
Later Rellich~\cite{Rellich1943} showed that condition~\eqref{eq:Sommerfeld} could be weakened to the integral form
\begin{equation}\label{eq:integralSommerfeld}
	\lim_{r\to \infty} \int_{\partial B(0,r)} |\nabla (e^{-i\sqrt{\kappa}|x|} u)|^2\,d\sigma(x)=0.
\end{equation}
Thus, again, estimate~\eqref{eq:lr} appears natural in virtue of condition~\eqref{eq:integralSommerfeld}. As a matter of fact, analogous estimates were needed also for proving results in other context, see for instance~\cite{Eidus62,IkebeSaito1972}, where a priori estimates for a suitable change of gauge were used to establish limiting absorption principle for suitable electromagnetic Helmholtz operators.
\end{remark}

\subsection{The main ideas}\label{Sec.sketch}
%
%\begin{remark}\label{rmk:sketch}
Here, we want to briefly comment on the strategy of the proof of Theorems~\ref{thm:main-nsa} and~\ref{thm:apriori}. 
As customarily in the proofs of such results, one treats differently the case of spectral parameter $z$ in the region containing the essential spectrum $[0, \infty),$ namely $\mathcal{S}_{\textup{pos}}:=\{ z\in \C\colon \Re z\geq 0\},$ or in the region $\mathcal{S}_\textup{neg}:=\{z\in \C\colon \Re z<0\}$ (see Fig.~\ref{fig:sectors-z}). 
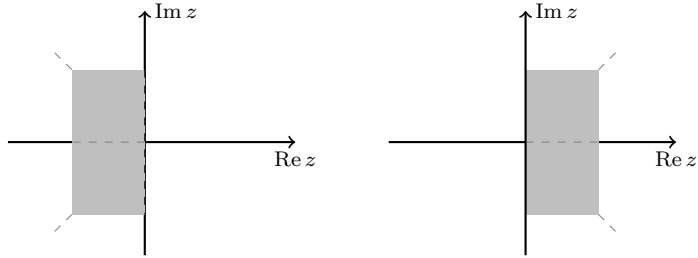
\begin{figure}[ht]\label{fig:sectors-z}
\centering
\begin{tikzpicture}[scale=.3]
\draw[thick,->](-6,0)--(6.6,0) node[below]{\footnotesize $\Re z$};
\draw[thick,->](0,-5)--(0,5.8) node[right]{\footnotesize $\Im z$};
\fill[gray!50] (0,3.2)--(-3.2,3.2)--(-3.2,-3.2)--(0,-3.2)--cycle;
\draw[gray, dashed] (0,0)--(-3.2,0);
\draw [gray, dashed] (-3.2,3.2)--(-4,4);
\draw [gray, dashed] (-3.2,-3.2)--(-4,-4);
\draw [dashed] (0,3.2)--(0,-3.2);
\end{tikzpicture}
\qquad
\begin{tikzpicture}[scale=.3]
\draw[thick,->](-6,0)--(6.6,0) node[below]{\footnotesize $\Re z$};
\fill[gray!50] (0,3.2)--(3.2,3.2)--(3.2,-3.2)--(0,-3.2)--cycle;
\draw[gray, dashed] (0,0)--(3.2,0);
\draw [gray, dashed] (3.2,3.2)--(4,4);
\draw [gray, dashed] (3.2,-3.2)--(4,-4);
\draw[thick,->](0,-5)--(0,5.8) node[right]{\footnotesize $\Im z$};
\end{tikzpicture}
\caption{The region $\mathcal{S}_\textup{neg}$ (left) and  $\mathcal{S}_\textup{pos}$ (right) in the $\Re z/\Im z$ plane} 
\end{figure}

Showing that there is no point spectrum in $\mathcal{S}_\textup{neg},$ or, more generally, proving resolvent estimates in this region, is easy and the proof is similar to the one of Theorem~\ref{thm:absence-out} and requires just a direct application of the method of multipliers suitably adapted to the biharmonic operator. 

On the other hand,
the proof in the region $\mathcal{S}_\textup{pos}$ is instead highly non-trivial: differently from the situation of second order operators, in the case of biharmonic operators a direct approach seems not suitably applicable. It seems instead successful using an  ``induction" argument that reduces the analysis to the study of two Schrödinger operators. This induction argument relies on the following natural decomposition
\begin{equation}\label{eq:decomposition-op}
	\Delta^2-z=(\Delta-\sqrt{z})(\Delta + \sqrt{z}),
\end{equation}
where $\sqrt{z}$ represents the principal square root of $z\in \C.$ The main advantage of this decomposition over the bilaplacian formulation is that it allows, even if not in a straightforward way, to use the existing extensive bibliography on Schrödinger operators. 

One should also notice that another advantage of the decomposition~\eqref{eq:decomposition-op} is that the spectral parameters in the two Schrödinger operators appear with different sign. This results in the fact that only one of the two will be troublesome. More precisely, if dealing with $\Delta + \sqrt{z}$ requires a more involved analysis 
as $\Delta + \sqrt{z}$ is not necessarily invertible in $\mathcal{S}_\textup{pos},$ indeed  $\sqrt{z}$ lies in the region shown in Figure~\ref{fig:sectors} (right) (since $\sqrt{z}$ is the principal square root, then $\Re \sqrt{z}\geq 0$), on the other hand, the operator $\Delta-\sqrt{z}$ is invertible (the spectral parameter appears with the opposite sign compared to the previous case), thus absence of eigenvalues or, more in general, resolvent estimates are easier to get.
\begin{figure}[ht]\label{fig:sectors}
\centering
\begin{tikzpicture}[scale=.3]
	\draw[thick,->](-6,0)--(6.6,0) node[below]{\footnotesize $\Re \sqrt{z}$};
% S_r
\fill[gray!50] (0,0)--(0,3.2)--(3.2,3.2)--(0,0)--cycle;
%\draw[gray, dashed] (0,0)--(0,3.2);
\draw (0,0)--(3,3)--cycle;
\draw (0,0)--(3,-3)--cycle;
\draw [dashed] (3,3)--(4,4);
\draw [dashed] (3,-3)--(4,-4);
%S_l
\fill[gray!50] (0,0)--(3.2,-3.2)--(0,-3.2)--(0,0)--cycle;
%\draw[gray, dashed] (0,0)--(0,-3.2);
%\draw (0,0)--(-3,3)--cycle;
%\draw (0,0)--(-3,-3)--cycle;
%\draw [dashed] (-3,3)--(-4,4);
%\draw [dashed] (-3,-3)--(-4,-4);
\draw[thick,->](0,-5)--(0,5.8) node[right]{\footnotesize $\Im \sqrt{z}$};
\end{tikzpicture}
\qquad
\begin{tikzpicture}[scale=.3]
	\draw[thick,->](-6,0)--(6.6,0) node[below]{\footnotesize $\Re \sqrt{z}$};
\draw[thick,->](0,-5)--(0,5.8) node[right]{\footnotesize $\Im \sqrt{z}$};
% S_r
\fill[gray!50] (0,0)--(3.2,3.2)--(3.2,-3.2)--(0,0)--cycle;
\draw[gray, dashed] (0,0)--(3.2,0);
%\draw (0,0)--(3,3)--cycle;
%\draw (0,0)--(3,-3)--cycle;
\draw [dashed] (0,0)--(4,4);
\draw [dashed] (0,0)--(4,-4);
%S_l
%\fill[gray!50] (0,0)--(-3.2,3.2)--(-3.2,-3.2)--(0,0)--cycle;
%\draw[gray, dashed] (0,0)--(-3.2,0);
%%\draw (0,0)--(-3,3)--cycle;
%%\draw (0,0)--(-3,-3)--cycle;
%\draw [dashed] (0,0)--(-4,4);
%\draw [dashed] (0,0)--(-4,-4);
\end{tikzpicture}
\caption{The region $\mathcal{S}_{\textup{neg}}:=\{0\leq \Re \sqrt{z} < |\Im \sqrt{z}|\}$ (left) and  $\mathcal{S}_\textup{pos}:=\{\Re \sqrt{z}\geq |\Im\sqrt{z}|\}$ (right) in the $\Re \sqrt{z}/\Im \sqrt{z}$ plane} 
\end{figure}
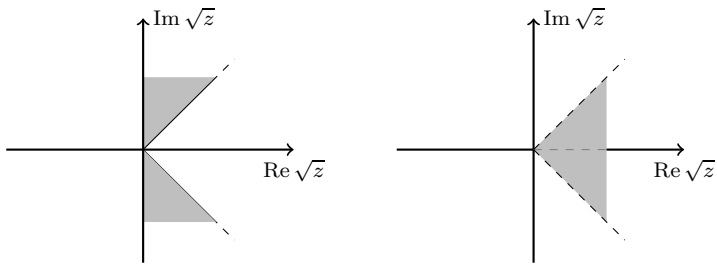
%\end{remark}

\subsection{More on self-adjoint perturbations}
In the self-adjoint case we have the following alternative results. The first one is concerned with detecting some natural \emph{repulsivity} condition for the potential guaranteeing absence of point spectrum of the perturbed self-adjoint biharmonic operator.
In passing, we should mention that self-adjoint biharmonic operators have attracted the interest of the mathematics community also in other contexts, see for instance~\cite{Feng-Soffer-Yao_2018,GreenToprak2019,ErdoganGreenToprak2021,Schlag2021,LiSofferYao2023,FengSofferWuYao2020,ErdoganGoldbergGreen2023,ErdoganGreen2022,ErdoganGreen2023,ErdoganGoldbergGreen2023-preprint} for several recent results on local and global dispersive estimates and absence of positive resonances. 
\begin{theorem}[Absence of eigenvalues: self-adjoint]
\label{thm:main-sa}
Let $d\geq 5.$ Suppose that $V\colon \R^d \to \R$ is such that $V\in L^1_\textup{loc}(\R^d)$ and $[x\cdot \nabla V]_+\in L^1_\textup{loc}(\R^d).$ Moreover, assume that there exists $a\in [0,1)$ such that
\begin{equation}\label{eq:sa}
	\forall\, \psi \in H^2(\R^d),
		\qquad \frac{1}{4}\int_{\R^d} [x\cdot \nabla V]_+ |\psi|^2\leq a \int_{\R^d} |\Delta u|^2. 
\end{equation}
If also $V\in W^{1,d/4}_\textup{loc},$ then $\sigma_\mathrm{p}(H_V)=\varnothing.$
\end{theorem}
In the next result we show related uniform resolvent estimate
in the self-adjoint setting. 
\begin{theorem}[Uniform resolvent estimates: self-adjoint]\label{thm:resolvent-est-sa}
	Let $d\geq 5.$ 
	Under the hypotheses of Theorem~\ref{thm:main-sa},
    for all $z\in \R,$
	\begin{equation*}
		\|r^{-2}(H_V-z)^{-1}r^{-2}\|_{L^2(\R^d)\to L^2(\R^d)} \leq c(d),
	\end{equation*}
	where $c(d)= \frac{1}{\sqrt{C_\textup{R}}}\frac{1}{1-a}\frac{2(d-2)}{d(d-4)},$ with $C_\textup{R}$ as in~\eqref{eq:Rellich}.
\end{theorem}
As in the non-self-adjoint case, Theorem~\ref{thm:resolvent-est-sa} can be obtained as a corollary of the following a priori estimates.
\begin{theorem}[A priori estimates: self-adjoint]\label{thm:apriori-sa}
	Let $d\geq 5.$ 
	Under the hypotheses of Theorem~\ref{thm:main-sa},
 given any $z\in \R$ and $r^2f\in L^2(\R^d),$ any solution $u\in H^2(\R^d)$ of the equation $(H_V-z)u=f$ satisfies
		\begin{equation}\label{eq:apriori}
			\|\Delta u\|_{L^2(\R^d)}\leq \tilde{c}(d) \|r^2 f\|_{L^2(\R^2)},
		\end{equation}
	where $\tilde{c}(d):= \frac{1}{1-a}\frac{2(d-2)}{d(d-4)},$ with $a$ as in~\eqref{eq:sa}.
\end{theorem}

\subsection{Organisation of the paper}
The paper is organised as follows. 
In Section~\ref{sec:identities} we collect the basic identities  
on which the technique of multipliers is based.
Section~\ref{sec:self-adjoint} is devoted to the proof of the results in the self-adjoint setting, 
namely Theorems~\ref{thm:main-sa}--\ref{thm:apriori-sa}.
Finally, in Section~\ref{sec:non-self-adjoint} we prove our main results valid for non-self-adjoint operators, 
namely Theorems~\ref{thm:main-nsa}--\ref{thm:apriori}.

%-----------------%
\section{Multipliers identities}\label{sec:identities}
%-----------------%

In this section we develop the method of multipliers which provides us with useful identities for suitable solutions of resolvent-type equations associated to the bilaplacian operator.

In the non-self-adjoint case we will need to develop the method for solutions of the following resolvent problem
\begin{equation}\label{eq:resolvent}
	\Delta^2u -zu=f,
\end{equation} 
where $z\in \C$ 
and $f\colon \R^d\to \C$ is measurable. 
We collect the identities and their proofs in the following lemma.

\begin{lemma}\label{lemma:identities}
	Let $u\in H^2(\R^d)$ be any solution of~\eqref{eq:resolvent} with $r^2 f\in L^2_\textup{loc}(\R^d).$
	The following identities hold true:
	\begin{align}
		\label{eq:S1}
		\int_{\R^d} |\Delta u|^2
%		-2\int_{\R^d} \nabla u D^2 g \overline{\nabla u}
%		+ \frac{1}{2}\int_{\R^d} \Delta^2 g|u|^2
		-\Re z\int_{\R^d} |u|^2
		&=\Re \int_{\R^d} f \bar{u},
		\tag{$\mathcal{S}_1$}\\
		\label{eq:S2}
		%\Im \int_{\R^d} \Delta u \Delta g \bar{u}
		%+ 2\Im \int_{\R^d} \Delta u \nabla g \cdot \overline{\nabla u}
		-\Im z \int_{\R^d} |u|^2
		&= \Im \int_{\R^d} f \bar{u}.
		\tag{$\mathcal{S}_2$}
%		\\
%		\label{eq:A1}		
%		4\int_{\R^d}|\Delta u|^2 + 2\Im z \Im \int_{\R^d} u x\cdot \overline{\nabla u}
%		&=\Re \int_{\R^d} f(2x\cdot \nabla + d)\bar{u}, 
%		\tag{$\mathcal{A}$}
%		\\
%		\label{eq:A2}
%		4\int_{\R^d} |\nabla \Delta u|^2 -2z_2\Im \int_{\R^d} \Delta u x\cdot \overline{\nabla u}
%		=-\Re \int_{\R^d} f(2x \cdot \nabla \Delta + (d+2)\Delta)\bar{u},
%		\tag{$A_2$}
%		\\
%		\label{eq:SK}
%		4\int_{\R^d} |\Delta u|^2
%		-\int_{\R^d} |x\cdot \nabla \Delta u|^2 
%		+ z_1\int_{\R^d} |x\cdot \nabla u|^2
%		=\Re \int_{\R^d} f (\nabla\cdot  x x\cdot  \nabla \bar{u})
%		\tag{$S_K$}\\
%		\label{eq:M1}
%		-\int_{\R^d} |\nabla \Delta u|^2 g + \frac{1}{2}\int_{\R^d} |\Delta u|^2 \Delta g 
%		+ z_1 \int |\nabla u|^2 g
%		-\frac{z_1}{2}\int_{\R^d} \Delta g |u|^2
%		-z_2 \Im \int_{\R^d} u \nabla g \overline{\nabla u}
%		=\Re \int_{\R^d} f g \overline{\Delta u}  
%		\tag{$M_1$}\\
%		\label{eq:M2}
%		-\Im \int_{\R^d} \nabla \Delta u \cdot \nabla g \Delta u
%		+ z_2 \int_{\R^d} |\nabla u|^2 g
%		-\frac{z_2}{2} \int_{\R^d} \Delta g |u|^2
%		+ z_1 \Im \int_{\R^d} u \nabla g \overline{\nabla u}	
%		= \Im \int_{\R^d} f g \overline{\Delta u}	
%		\tag{$M_2$}
		\end{align}
		\end{lemma}
		\begin{proof}

			Identity~\eqref{eq:S1} is obtained multiplying~\eqref{eq:resolvent} by the symmetric multiplier $v:=\bar{u}\in H^2(\R^d),$ 
			integrating over $\R^d,$
			 integrating by parts, taking the real part of the resulting identity and integrating by parts again. 
			 Identity~\eqref{eq:S2} is obtained again by multiplication by $v:=\bar{u}\in H^2(\R^d),$ integrating over $\R^d,$ integrating by parts and taking the imaginary part of the resulting identity.  
\end{proof}

In the self-adjoint case we will need an alternative identity for suitable \emph{compactly supported} solutions of a more general resolvent-type equation associated to the bilaplacian operator, namely
\begin{equation}\label{eq:resolvent-general}
\Delta^2 u + Vu - zu=f + g,
\end{equation} 
where $z\in \R,$ and $f,g\colon \R^d \to \C$ are measurable functions. The identity that we will need and its proof is contained in the following lemma.
\begin{lemma}\label{lemma:A}
	Let $u\in H^2(\R^d)$ be any compactly supported solution of~\eqref{eq:resolvent-general}. Assume $r^2f\in L^2_{\textup{loc}}(\R^d)$ and $g\in L^2_{\textup{loc}}(\R^d),$ moreover $V\in W^{1,d/4}_{\textup{loc}}(\R^d).$ Then one has 
	 \begin{equation}
		\label{eq:A1}		
		4\int_{\R^d}|\Delta u|^2
		-\int_{\R^d} x\cdot \nabla V|u|^2
		=\Re \int_{\R^d} f (2x\cdot \nabla + d)\bar{u}
		+ \Re \int_{\R^d} g (2x\cdot \nabla + d)\bar{u}.
		\tag{$\mathcal{A}$}
	 \end{equation}
\end{lemma}

\begin{proof}
Identity~\eqref{eq:A1} is \emph{formally} obtained multiplying~\eqref{eq:resolvent-general} by $v:=A \bar{u},$ with $A$ being the skew-symmetric first order operator $A:=x\cdot \nabla + \nabla\cdot x=2x\cdot \nabla + d$
			(generator of dilation up to a constant), 
			integrating the identity over $\R^d,$ integrating by parts, taking the real part of the resulting identity and integrating by parts again.
			
			Nevertheless, since $2x\cdot \nabla u$ does not necessarily belong to $H^2(\R^d),$ the test function $v$ defined above might not be an admissible test function. Thus, for a rigorous proof of~\eqref{eq:A1}, a suitable regularisation argument is needed: Motivated by \cite{CossettiKrejcirik20}, we consider the following alternative test function $v_{\delta}:= d u +x \cdot [\nabla^\delta + \nabla^{-\delta}]u:= d u + x_k [\partial_k^\delta + \partial_k^{-\delta}]u$ where
\begin{equation*}
	\partial_k^\delta u(x):= \frac{\tau_k^\delta u(x)-u(x)}{\delta},
	\qquad \text{with} \quad \tau_k^\delta u(x):=u(x+ \delta e_k),
	\qquad k=1, 2, \dots, d,
\end{equation*}
with $\delta \in \R \setminus \{0\}$ is the standard \emph{difference quotient} of $u.$ Here and in the following we use the Einstein summation convention for repeated indices. Since $u$ is compactly supported and since we have replaced the standard gradient with the difference quotient, it is now clear that $v_{\delta}$ as defined above belongs to $H^2(\R^d)$ and therefore can be used as a test function in the weak formulation of~\eqref{eq:resolvent-general}. Then we take the real part of the resulting identity obtaining
\begin{equation}\label{eq:weak-real}
	\Re \langle \Delta u, \Delta v_{\delta} \rangle
	+ \Re \langle Vu, v_\delta \rangle 
	=\Re (z\langle u, v_{\delta}\rangle)
	+\Re \langle f, v_{\delta} \rangle
	+\Re \langle g, v_{\delta} \rangle.
\end{equation}
Now we treat each term in~\eqref{eq:weak-real} separately.

We start from the kinetic contribution. Namely
\begin{equation*}
	\Re \langle \Delta u, \Delta v_{\delta} \rangle
	=d \int_{\R^d} |\Delta u|^2 
	+ \Re \int_{\R^d} \Delta u x_k [\partial_k^\delta + \partial_k ^{-\delta}] \Delta \overline{u}
	+ 2 \Re \int_{\R^d} \Delta u [\partial_k^\delta + \partial_k ^{-\delta}] \partial_k \overline{u}.
\end{equation*}
Using the identity
\begin{equation}\label{eq:Re-diff}
	2\Re(\overline{\psi} \partial_k^\delta \psi)=\partial_k^\delta |\psi|^2 - \delta |\partial_k^\delta \psi|^2,
	\qquad k=1,2, \dots, d,
\end{equation}
valid for every $\psi\colon \R^d \to \C,$ one has
\begin{multline*}
	\Re \langle \Delta u, \Delta v_{\delta} \rangle
	=d \int_{\R^d} |\Delta u|^2 
	+ \frac{1}{2} \int_{\R^d} x_k [\partial_k^\delta + \partial_k ^{-\delta}] |\Delta u|^2\\
	-\frac{\delta}{2}\int_{\R^d} x_k \left(|\partial_k^\delta \Delta u|^2 - |\partial_k^{-\delta} \Delta u|^2 \right)
	+ 2 \Re \int_{\R^d} \Delta u [\partial_k^\delta + \partial_k ^{-\delta}] \partial_k \overline{u}.
\end{multline*}
Integrating by parts in the second term of the right-hand-side of the previous identity and making explicit the difference quotient and changing variable in the third term gives
\begin{equation*}
	\Re \langle \Delta u, \Delta v_{\delta} \rangle
	= \frac{1}{2} \int_{\R^d} |\tau_k^\delta \Delta u - \Delta u|^2
	+ 2\Re \int_{\R^d} \Delta u [\partial_k^\delta + \partial_k ^{-\delta}] \partial_k \overline{u}.
\end{equation*} 
Here we have used the integration by parts formula for different quotients
\begin{equation}\label{eq:intparts}
	\int_{\R^d} \varphi \partial_k^{\pm \delta} \psi
	= - \int_{\R^d} (\partial_k^{\mp} \varphi) \psi
\end{equation} 
which holds true for every $\varphi, \psi \in L^2(\R^d)$ (see~\cite[Sec. 5.8.2]{Evans-book}).
Passing to the limit as $\delta$ goes to zero one gets
\begin{equation}\label{eq:K-limit}
	\Re \langle \Delta u, \Delta v_{\delta} \rangle
	\xrightarrow{\delta \to 0}
	4\int_{\R^d} |\Delta u|^2,
\end{equation}
where here we have used the $L^2$-continuity of the translations and the strong $L^2$-convergence of the difference quotients to standard derivatives.

Now we consider the term depending on the spectral parameter. Since $z\in \R$ one has
\begin{equation*}
	\Re \left ( z \langle u, v_{\delta} \rangle\right)
	=z d \int_{\R^d}|u|^2 
	+ z \int_{\R^d} x_k \Re \left(u [\partial_k^\delta + \partial_k^{-\delta}] \overline{u}\right).
%	- \Im z \Im \int_{\R^d} x_k u [\partial_k^\delta + \partial_k^{-\delta}] \overline{u}.
\end{equation*}
Using~\eqref{eq:Re-diff} and the integration by parts formula~\eqref{eq:intparts} gives
\begin{equation*}
	\Re \left ( z \langle u, v_{\delta} \rangle\right)
	=\frac{z}{2} \int_{\R^d} |\tau_k^\delta u-u|^2.
%	- \Im z \Im \int_{\R^d} x_k u [\partial_k^\delta + \partial_k^{-\delta}] \overline{u}.
\end{equation*}
Using the $L^2$-continuity of the translations on the right-hand-side one gets
\begin{equation}\label{eq:spectral-limit}
	\Re(z\langle u, v_{\delta} \rangle)
	\xrightarrow{\delta \to 0}
	0.
%	-2\Im z \Im \int_{\R^d} x_k u_R \partial_k \overline{u_R}.
\end{equation}

Now we consider the term depending on $V,$ namely
\begin{equation*}
	\Re \langle Vu, v_{\delta} \rangle 
	=d \int_{\R^d} V |u|^2 
	+ \Re \int_{\R^d} V u x_k [\partial_k^\delta + \partial_k^{-\delta}] \overline{u}.
\end{equation*}
We will show that 
\begin{equation}\label{eq:limit-V}
	\Re \langle Vu, v_{\delta} \rangle 
	\xrightarrow{\delta\to 0}
	-\int_{\R^d} x\cdot \nabla V |u|^2.
\end{equation}
Using~\eqref{eq:Re-diff} and the integration by parts formula~\eqref{eq:intparts} gives
\begin{equation*}
	\Re \int_{\R^d} V u x_k [\partial_k^\delta + \partial_k^{-\delta}] \overline{u}
	=-\frac{1}{2} \int_{\R^d} [\partial_k^\delta + \partial_k^{-\delta}] (x_k V) |u|^2
	+ \frac{1}{2} \int_{\R^d} \partial_k^\delta (x_k V) |\tau_k^\delta u-u|^2.
\end{equation*}
Using H\"older inequality one has
\begin{equation*}
	\left|
	\int_{\R^d} (\partial_k^\delta -\partial_k)(x_k V) |u|^2	\right|
	\leq \|(\partial_k^\delta -\partial_k)(x_k V)\|_{d/4} \|u\|_{2d/(d-4)}
		\xrightarrow{\delta \to 0} 0,
\end{equation*}
indeed, since $x_k V\in W^{1,d/4}_{loc},$ the term $\|(\partial_k^\delta -\partial_k)(x_k V)\|_{d/4}$ goes to zero as $\delta$ goes to zero (using the $L^p$-convergence of the difference quotients), moreover $\|u\|_{2d/(d-4)}$ is finite by Sobolev's embeddings. Similarly
\begin{equation*}
	\left|
	\int_{\R^d} \partial_k^\delta (x_k V) |\tau_k^\delta u-u|^2	\right|
	\leq \|\partial_k^\delta (x_k V)\|_{d/4} \|\tau_k^\delta u-u\|_{2d/(d-4)}
		\xrightarrow{\delta \to 0} 0.
\end{equation*}
From this we get
\begin{equation*}
	\Re \int_{\R^d} V u x_k [\partial_k^\delta + \partial_k^{-\delta}] \overline{u} 
	\xrightarrow{\delta \to 0}
	-\int_{\R^d} \partial_k(x_k V)|u|^2
	=-d\int_{\R^d} V |u|^2
	-\int_{\R^d} x\cdot \nabla V |u|^2.
\end{equation*}
From this,~\eqref{eq:limit-V} follows immediately.

We consider now the term depending on the source term $f.$ Namely
\begin{equation*}
	\Re \langle f, v_{\delta} \rangle 
	=d \Re \int_{\R^d} f \overline{u} 
	+ \Re \int_{\R^d} f x_k [\partial_k^\delta + \partial_k^{-\delta}] \overline{u}.
\end{equation*}
We want to show that under the hypothesis on $f$ one has
\begin{equation}\label{eq:limit-f}
	\Re \langle f, v_{\delta} \rangle 
	\xrightarrow{\delta\to 0}
	d \Re \int_{\R^d} f \overline{u} +
	2\Re \int_{\R^d} f x_k \partial_k \overline{u}.
\end{equation}
In order to show~\eqref{eq:limit-f} we need the following lemma which shows that under some higher order regularity assumptions the difference quotients converge to the standard derivatives also for suitable weighted $L^2$ spaces.
\begin{lemma}\label{lemma:f-diff}
	Let $d\geq 5$ and $u\in H^2(\R^d).$ Then the following weighted-$L^2$-convergence holds
	\begin{equation}\label{eq:f-diff}
		\left\|\frac{\partial_k^\delta u - \partial_k u}{|x|}\right\|_2
		\xrightarrow{\delta\to 0}
		 0,
		 \qquad k\in \{1,2,\dots,d\}.
	\end{equation} 
\end{lemma}
\begin{proof}
First of all one notices that using the fundamental theorem of calculus one has
\begin{equation}\label{eq:ftc}
	\partial_k^\delta u(x)=\int_0^1 \partial_k u(x+\delta te_k) dt.
\end{equation}
Using~\eqref{eq:ftc} we have
\begin{equation*}
	\begin{split}
		\left\|\frac{\partial_k^\delta u - \partial_k u}{|x|}\right\|_2^2
		&=\int_{\R^d} \frac{1}{|x|^2} \left|\int_0^1 \partial_k u(x+\delta t e_k) - \partial_k u(x) \,dt \right|^2 \, dx\\
		&\leq \int_{\R^d} \frac{1}{|x|^2} \int_0^1 |\partial_k u(x+\delta t e_k) - \partial_k u(x)|^2\, dt \, dx\\
		&\leq \int_0^1 \int_{\R^d} \frac{| \nabla  \left[u(x+\delta t e_k)- u(x)\right]|^2}{|x|^2}\, dt\, dx\\
		&\leq \frac{1}{C_\textup{HR}} \int_0^1 \int_{\R^d} |\Delta u(x+\delta t e_k) - \Delta u(x)|^2\, dt \, dx
		=:\int_0^1 f_\delta(t)\, dt,
	\end{split}
\end{equation*}
where $C_{\textup{HR}}$ is as in~\eqref{eq:HR}.
Observe that $f_\delta (t)\leq 4 \|\Delta u\|_2^2,$ moreover $f_\delta(t)$ goes to zero as $\delta$ goes to zero, due to the $L^2$-continuity of the translations. From this, the thesis follows as a consequence of the Lebesgue convergence theorem.
\end{proof}
With Lemma~\ref{lemma:f-diff} at hands we can now prove~\eqref{eq:limit-f}. One has
\begin{equation*}
	\begin{split}
	\left|\int_{\R^d} f x_k (\partial_k^\delta -\partial_k)\overline{u} \right|
	&\leq \int_{\R^d} |x|^2|f| \left|\frac{(\partial_k^\delta -\partial_k)\overline{u}}{|x|}\right| \\
	&\leq \||x|^2 f\|_2 \left \| \frac{\partial_k^\delta u- \partial_k u}{|x|}\right\|_2.
	\end{split}
\end{equation*}
Since $|x|^2 f\in L^2_\text{loc}(\R^d),$ the limit~\eqref{eq:limit-f} follows using~\eqref{eq:f-diff} in Lemma~\ref{lemma:f-diff}.

It is left to treat the term depending on $g,$ namely
\begin{equation*}
	\Re \langle g, v_{\delta} \rangle 
	=d \Re \int_{\R^d} g \overline{u} 
	+ \Re \int_{\R^d} g x_k [\partial_k^\delta + \partial_k^{-\delta}] \overline{u}.
\end{equation*}
 One has the following limits
\begin{equation}\label{eq:limit-g}
	\Re  \langle g, \nabla v_{\delta} \rangle
	\xrightarrow{\delta \to 0} 
	d\Re \int_{\R^d}	g\overline{u}
	+2\Re \int_{\R^d} g x_k \partial_k \overline{u}.
\end{equation}
Indeed, since $g$ belongs to $L^2(\R^d),$ the limit above follows just using Cauchy Schwarz and the $L^2$-convergence of the difference quotients (notice that the unbounded weight $x_k$ does not affect the convergence since $u$ is compactly supported).

Passing to the limit in~\eqref{eq:weak-real} using~\eqref{eq:K-limit},~\eqref{eq:spectral-limit},~\eqref{eq:limit-V},~\eqref{eq:limit-f} and~\eqref{eq:limit-g} one obtains~\eqref{eq:A1}.
\end{proof}

\begin{remark}\label{Rem.Mourre}
We remark that identity~\eqref{eq:A1} in Lemma~\ref{lemma:identities} 
is closely related to the commutator theory \emph{\`a la} Mourre
of conjugate operators 
(see~\cite{Mourre80} for the pioneering work and~\cite{Amrein1996} for more recent developments)
and the virial theorem in quantum mechanics
(see~\cite{Weidmann1967} for a first rigorous treatment 
and~\cite[Sec.~XIII.13]{ReedSimonIV} for an overview and further references).
Indeed, one can easily see that the following identity holds
\begin{equation*}
	2\Re\langle \Delta^2 u, Au \rangle = \langle u, [\Delta^2, A]u \rangle, 
\end{equation*}
with $A$ as above, namely $A= x\cdot\nabla+\nabla\cdot x.$ Thus, the multiplier method used above to prove identity~\eqref{eq:A1} is related to the computation of a commutator as in the standard positive commutator theory. 
If the Mourre theory for Schrödinger operators has been extensively studied and many results have been obtained using this method in many areas of spectral theory, for biharmonic operators the bibliography is much scarcer. Nonetheless there are works using this approach to prove, for instance, Jensen-Kato type decay estimates for the evolution operator (see~\cite{Feng-Soffer-Yao_2018}).  
\end{remark}

\section{Self-adjoint setting}\label{sec:self-adjoint}
This section is concerned with the proof of the results in the self-adjoint case, namely Theorems~\ref{thm:main-sa}--\ref{thm:apriori-sa}. 

In the proof of Theorem~\ref{thm:main-sa} we will need the following result.
\begin{lemma}\label{lemma:harmonic-zero}
	Under the assumption of Theorem~\ref{thm:main-sa}, 
	let $u\in H^2(\R^d)$ be any solution of
\begin{equation}\label{eq:eigenv-eq}
	H_Vu=zu,
\end{equation}
If $\Delta u=0$ then $u=0.$ 
\end{lemma}
\begin{proof}
We present a proof which does not require to 
argue by Liouville's theorem.
	Since $\Delta u=0,$ in particular one has that $u$ satisfies (weakly)
	\begin{equation} \label{eq:deltau-zero}
		Vu=zu. 
	\end{equation}	
	From this we want to show that $u=0.$ We distinguish two cases:
	\begin{itemize}
		\item If $z \neq 0.$ We multiply~\eqref{eq:deltau-zero} by $\overline{u}$ and we integrate over $\R^d$ obtaining
		\begin{equation*}
			\int_{\R^d} V |u|^2 = z \int_{\R^d} |u|^2.
		\end{equation*}
		In particular one has
		\begin{equation*}
			\begin{split}
				|z| \int_{\R^d} |u|^2&\leq \int_{\R^d} |V||u|^2\\
				&\leq \|V\|_{d/4}\|u\|_{2d/(d-4)}^2\\
				&\leq C(d)\|V\|_{d/4}\|\Delta u\|_2\\
				&=0,
			\end{split}
		\end{equation*}
		where we used, in order, H\"older inequality and Sobolev embeddings for homogeneous Sobolev spaces. In particular $C(d)$ is the constant in the Sobolev embeddings. This gives
		\begin{equation*}
			|z| \int_{\R^d} |u|^2\leq 0,
		\end{equation*}
		that implies $u=0.$
	\item If $z=0.$ If $V=0,$ then $u$ solves $\Delta^2 u=0$ which implies $u=0$ as $\sigma_\textup{p}(\Delta^2)=\varnothing.$ So we can assume that $V\neq 0.$ Suppose by contradiction that $u\neq 0,$ thus $\psi:=|V|^{1/2} u$ is also non trivial. Indeed, since $u$ solves $\Delta^2 u + Vu=0,$ in particular one has $\Delta^2 u=- V |V|^{-1/2} |V|^{1/2} u.$ Thus if $\psi$ were trivial then $0\in \sigma_\textup{p}(\Delta^2)$ that is a contradiction.  Nevertheless, using the computations of the previous case one gets
	\begin{equation*}
		\int_{\R^d} |V||u|^2
		\leq C(d)\|V\|_{d/4}\|\Delta u\|_2
		=0.
	\end{equation*}
	Thus $\psi:=|V|^{1/2}u=0$ which is a contradiction. This concludes the proof.
	\qedhere
	\end{itemize} 
\end{proof}

\begin{proof}[Proof of Theorem~\ref{thm:main-sa}]
Assume by contradiction that there exists a non trivial (weak) solution $u\in H^2(\R^d)$ to the eigenvalue equation~\eqref{eq:eigenv-eq}
with $H_V$ defined in~\eqref{eq:biharmonic-op}. Using the method of multipliers we will show that $\Delta u=0$ and therefore, using Lemma~\ref{lemma:harmonic-zero}, $u=0.$ 

Similarly to the case of self-adjoint Schrödinger operators~\cite[Thm. 3.4]{CossettiFanelliKrejcirik20}, the triviality of solution of~\eqref{eq:eigenv-eq} will be obtained from the single identity~\eqref{eq:A1}.

In order to use~\eqref{eq:A1} we need to approximate $u$ by a compactly supported function. Let $\mu \colon [0,\infty)\to [0,1]$ be a smooth function such that 
\begin{equation}\label{eq:mucutoff}
\mu(r)=
	\begin{cases}
		\; 1 \quad \text{if} \quad 0\leq r\leq 1,\\
		\; 0 \quad \text{if}\quad  r\geq 2.		
	\end{cases}	
\end{equation}
Given a positive number $R,$ we set $\mu_R(x):=\mu(|x|R^{-1}).$ Then $\mu_R\colon \R^d\to [0,1]$ is such that 
\begin{equation*}
	\mu_R=1 \text{ in } B_R(0),
	\quad \mu_R=0 \text{ in } \R^d\setminus B_{2R}(0),
	\quad |\nabla \mu_R|\leq cR^{-1},
	\quad |\Delta \mu_R|\leq c R^{-2},
\end{equation*}
where $B_{R}(0)$ stands for the open ball centred at the origin and with radius $R>0$ and $c>1$ is a suitable constant independent of $R.$ Now we define the approximating family of compactly supported functions $u_R:=u \mu_R.$ Since $u$ satisfies weakly equation~\eqref{eq:eigenv-eq}, one can easily check that $u_R$ defined above satisfies
\begin{equation}\label{eq:weak-uR}
	\langle \Delta u_R, \Delta v \rangle
	+\langle V u_R, v \rangle
	=z\langle u_R, v\rangle
	+ \langle \text{err}^{(1)}(R), v \rangle
	+ \langle \text{err}^{(2)}(R), \nabla v\rangle,
	\qquad \forall \, v\in H^2(\R^d),
\end{equation}
where
\begin{equation}\label{eq:err12}
	\text{err}^{(1)}(R)
	:= \Delta^2 \mu_R u 
	+ 4 \nabla \Delta \mu_R\nabla u 
	+2 \nabla \partial_j \mu_R \nabla \partial_j u,
	\qquad \text{and} \qquad
	\text{err}^{(2)}(R)
	:=-2\partial_j \mu_R \nabla \partial_j u 
	-2 \nabla \mu_R \Delta u,
\end{equation}
for $j\in \{1,2,\dots, d\}.$ Here, as above, we have used the usual convention that repeated indices are \soutD{implicitly} summed over.

Since $u_R\in H^2(\R^d)$ is compactly supported and satisfies~\eqref{eq:resolvent-general} with $f=0$ and $g= \text{err}^{(1)}(R) + \nabla \cdot  \text{err}^{(2)}(R)$ we can use~\eqref{eq:A1} obtaining
\begin{multline}\label{eq:weak-real-man}
	4\int_{\R^d} |\Delta u_R|^2
	-\int_{\R^d} [x\cdot \nabla V]_+ |u_R|^2
	\leq
	+d\Re \int_{\R^d}	\text{err}^{(1)}(R)\overline{u_R}
	+2\Re \int_{\R^d}\text{err}^{(1)}(R) x_k \partial_k \overline{u_R}\\
	+d \Re \int_{\R^d} \text{err}^{(2)}(R) \nabla \overline{u_R}
	+2\Re \int_{\R^d} \text{err}^{(2)}(R) e_k \partial_k \overline{u_R} 
	+2\Re \int_{\R^d} \text{err}^{(2)}(R) x_k \partial_k \nabla  \overline{u_R},
\end{multline}
where the inequality comes from having discarded the negative part of the radial derivative of $V,$ namely $[x\cdot \nabla V]_-.$

Now we want to pass to the limit $R$ to infinity. First of all notice that, simply using the definition of the cut-off $\mu_R,$ one has that $u_R$ converges to $u$ in $H^2(\R^d)$ in the limit $R$ to infinity.

Now we want to show that  
\begin{equation}\label{eq:error-L2}
\|\text{err}^{(k)}(R)\|_2\xrightarrow{R\to \infty} 0,
\quad \text{and} \quad 
\||x|\text{err}^{(k)}(R)\|_2\xrightarrow{R\to \infty} 0, 
\quad k\in \{1,2\}.
\end{equation} 
We start with $\text{err}^{(1)}(R).$ Using definition~\eqref{eq:err12} it is enough to show that 
\begin{equation*}
	\text{I}:=\int_{\R^d} |\Delta^2 \mu_R|^2 |u|^2
	+\int_{\R^d} |\nabla \Delta \mu_R|^2 |\nabla u|^2
	+\int_{\R^d} |\nabla \partial_j \mu_R|^2 |\nabla \partial_j u|^2
	\xrightarrow{R\to \infty} 0,
\end{equation*}
and 
\begin{equation*}
	\text{I}_{|x|}:=\int_{\R^d} |x|^2 |\Delta^2 \mu_R|^2 |u|^2
	+\int_{\R^d} |x|^2 |\nabla \Delta \mu_R|^2 |\nabla u|^2
	+\int_{\R^d} |x|^2 |\nabla \partial_j \mu_R|^2 |\nabla \partial_j u|^2
	\xrightarrow{R\to \infty} 0,
\end{equation*}
for $j\in\{1,2,\dots, d\}.$
Simply using the properties of the cut-off $\mu_R$ one has
\begin{equation*}
	\text{I}\lesssim
	\frac{1}{R^8} \int_{\R^d} |u|^2
	+\frac{1}{R^6} \int_{\R^d} |\nabla u|^2
	+\frac{1}{R^4} \int_{\R^d} |\nabla \partial_j u|^2
	\xrightarrow{R\to \infty} 0,
\end{equation*}
since $u\in H^2(\R^d).$ Similarly, recalling that since only derivatives of $\mu_R$ appear then $R\leq |x|\leq 2R,$ one has
\begin{equation*}
	\text{I}_{|x|}\lesssim
	\frac{R^2}{R^8} \int_{\R^d} |u|^2
	+\frac{R^2}{R^6} \int_{\R^d} |\nabla u|^2
	+\frac{R^2}{R^4} \int_{\R^d} |\nabla \partial_j u|^2
	\xrightarrow{R\to \infty} 0.
\end{equation*}
We continue with $\text{err}^{(2)}(R).$ Using~\eqref{eq:err12} it is enough to show that
\begin{equation*}
	\text{II}:=\int_{\R^d} |\partial_j \mu_R|^2 |\nabla \partial_j u|^2 
	+ \int_{\R^d} |\nabla \mu_R|^2 |\Delta u|^2
	\xrightarrow{R\to \infty} 0,
\end{equation*}
and
\begin{equation*}
	\text{II}_{|x|}:=\int_{\R^d} |x|^2 |\partial_j \mu_R|^2 |\nabla \partial_j u|^2 
	+ \int_{\R^d} |x|^2 |\nabla \mu_R|^2 |\Delta u|^2
	\xrightarrow{R\to \infty} 0,
\end{equation*}
for $j\in {1,2,\dots, d}.$ Similarly as in the case of $\text{err}^{(1)}(R)$ one has
\begin{equation*}
\text{II}
\lesssim 
\frac{1}{R^2}\int_{\R^d} |\nabla \partial_j u|^2 
+ \frac{1}{R^2}\int_{\R^d} |\Delta u|^2 
\xrightarrow{R\to \infty} 0.
\end{equation*}
Moreover, as above, we have
\begin{equation*}
	\begin{split}
	\text{II}_{|x|}
&\lesssim 
\frac{R^2}{R^2}\int_{R\leq |x|\leq 2R} |\nabla \partial_j u|^2 
+ \frac{R^2}{R^2}\int_{R\leq |x|\leq 2R} |\Delta u|^2 \\
&\leq 
\int_{|x|\geq R} |\nabla \partial_j u|^2 
+ \int_{ |x|\geq R} |\Delta u|^2 
\xrightarrow{R\to \infty} 0,
	\end{split}
\end{equation*}
since $u\in H^2(\R^d).$ From this one gets~\eqref{eq:error-L2}. From~\eqref{eq:error-L2} and from the $H^2$-convergence of $u_R$ to $u,$ using the Cauchy-Schwarz inequality it follows that the right-hand-side of~\eqref{eq:weak-real-man} goes to zero as $R$ goes to $+\infty.$ As for the left-hand-side of~\eqref{eq:weak-real-man} using for the first term  the $H^2$-convergence of $u_R$ to $u$ and for the second term the monotone convergence theorem one obtains

\begin{equation*}
4\int_{\R^d} |\Delta u|^2 - \int_{\R^d} [x\cdot \nabla V]_+ |u|^2\leq 0.
\end{equation*}
Using~\eqref{eq:sa} one gets $\int_{\R^d}|\Delta u|^2=0,$ which gives $u=0.$
\end{proof}

The proof of Theorem~\ref{thm:resolvent-est-sa} is a direct consequence of Theorem~\ref{thm:apriori-sa}, 
thus we prove the latter first.
\begin{proof}[Proof of Theorem~\ref{thm:apriori-sa}]
Let $u\in H^2(\R^d)$ be a solution of the resolvent equation 
\begin{equation}\label{eq:resolventV}
	H_V u -zu=f.
\end{equation}
Similarly to the proof of the previous result we introduce the compactly supported approximating functions $u_R:=u \mu_R,$ with $\mu_R(x):=\mu(|x|R^{-1})$ and $\mu$ as in~\eqref{eq:mucutoff}. Since $u$ satisfies weakly~\eqref{eq:resolventV}, one can easily check that $u_R$ satisfies
 \begin{equation}\label{eq:weak-uR-f}
	\langle \Delta u_R, \Delta v \rangle
	+\langle V u_R, v \rangle
	=z\langle u_R, v\rangle
	+ \langle f_R, v\rangle
	+ \langle \text{err}^{(1)}(R), v \rangle
	+ \langle \text{err}^{(2)}(R), \nabla v\rangle,
	\qquad \forall \, v\in H^2(\R^d),
\end{equation}
where $\text{err}^{(1)}(R)$ and $\text{err}^{(2)}(R)$ are as in~\eqref{eq:err12}. Thus, since $u_R\in H^2(\R^d)$ is compactly supported and satisfies~\eqref{eq:resolvent-general} with $f=f_R$ and $g= \text{err}^{(1)}(R) + \nabla \cdot  \text{err}^{(2)}(R)$ we can use~\eqref{eq:A1} obtaining
\begin{multline}\label{eq:weak-real-man-f}
	4\int_{\R^d} |\Delta u_R|^2
	-\int_{\R^d} [x\cdot \nabla V]_+ |u_R|^2
	\leq
	+d\Re \int_{\R^d}	f_R\overline{u_R}
	+2\Re \int_{\R^d} f_R x_k \partial_k \overline{u_R}
	\\
	+d\Re \int_{\R^d}	\text{err}^{(1)}(R)\overline{u_R}
	+2\Re \int_{\R^d}\text{err}^{(1)}(R) x_k \partial_k \overline{u_R}\\
	+d \Re \int_{\R^d} \text{err}^{(2)}(R) \nabla \overline{u_R}
	+2\Re \int_{\R^d} \text{err}^{(2)}(R) e_k \partial_k \overline{u_R} 
	+2\Re \int_{\R^d} \text{err}^{(2)}(R) x_k \partial_k \nabla  \overline{u_R},
\end{multline}
where the inequality comes from having discarded the negative part of the radial derivative of $V,$ namely $[x\cdot \nabla V]_-.$

Now we want to pass to the limit $R$ to infinity. We need just to check the terms depending on $f_R$ as the others were already considered in the proof of Theorem~\ref{thm:main-nsa} above. We will show that 
\begin{equation*}
	d\Re \int_{\R^d}	f_R\overline{u_R}
	+2\Re \int_{\R^d} f_R x_k \partial_k \overline{u_R}
	\xrightarrow{R\to\infty}
	d\Re \int_{\R^d}	f\overline{u}
	+2\Re \int_{\R^d} f x_k \partial_k \overline{u}.
\end{equation*}	
Indeed 
\begin{equation*}
	\begin{split}
	\left | \int_{\R^d} f_R \overline{u_R} -f \overline{u} \right|
	&\leq \int_{\R^d} |f_R-f| |u_R| 
	+ \int_{\R^d} |f| |u_R-u|\\
	&\leq \||x|^2 f (\mu_R-1)\|_2 \left\|\frac{u_R}{|x|^2}\right\|_2
	+ \||x|^2 f \|_2 \left\|\frac{u_R-u}{|x|^2}\right\|_2\\
	&\leq \frac{1}{\sqrt{C_\text{R}}} \||x|^2 f (\mu_R-1)\|_2 \|\Delta u_R\|_2
	+\frac{1}{\sqrt{C_\text{R}}} \||x|^2 f\|_2 \|\Delta (u_R-u)\|_2\\
	&\phantom{\leq}\xrightarrow{R\to \infty} 0,
	\end{split}
\end{equation*}
where $C_\text{R}$ is as in~\eqref{eq:Rellich} and where the limit follows from the Lebesgue convergence theorem and the $H^2$-convergence of $u_R$ to $u.$ Similarly we have
\begin{equation*}
\begin{split}
\left | \int_{\R^d} f_Rx_k \partial_k \overline{u_R} - fx_k \partial_k \overline{u} \right |
&\leq \int_{\R^d} |x||f_R-f||\nabla u_R| + \int_{\R^d} |x||f||\nabla (u_R -u)|\\
&\leq \||x|^2 f (\mu_R-1)\|_2 \left\|\frac{\nabla u_R}{|x|}\right\|_2
+ \||x|^2 f\|_2 \left\|\frac{\nabla (u_R- u)}{|x|}\right\|_2\\
&\leq \frac{1}{\sqrt{C_\text{HR}}} \||x|^2 f (\mu_R-1)\|_2 \|\Delta u_R\|_2
 + \frac{1}{\sqrt{C_\text{HR}}} \||x|^2 f\|_2 \|\Delta (u_R-u)\|_2\\
	&\phantom{\leq}\xrightarrow{R\to \infty} 0,
\end{split}
\end{equation*}
where $C_\text{R}$ is as in~\eqref{eq:HR} and where the limit follows from the Lebesgue convergence theorem and the $H^2$-convergence of $u_R$ to $u.$ Thus, passing to the limit $R$ to infinity in~\eqref{eq:weak-real-man-f} one gets

	\begin{equation}\label{eq:first}
		4\int_{\R^d} |\Delta u|^2 
		-\int_{\R^d} [x\cdot \nabla V]_+|u|^2		
		\leq 2\Re \int_{\R^d} f x\cdot \overline{\nabla u} + d \Re \int_{\R^d} f \bar{u}.
	\end{equation} 
	We leave for now the first term of the right-hand-side of~\eqref{eq:first}. As for the second term, using the Cauchy-Schwarz inequality and~\eqref{eq:HR} one gets
	\begin{equation}\label{eq:rhsI}
		%\begin{split}
		\Big| 2\Re \int_{\R^d} f x\cdot \overline{\nabla u} \Big|
		\leq 2\||x|^2 f\|_{2} \left \|\frac{\nabla u}{|x|}\right \|_2
		\leq \frac{4}{d}   \||x|^2 f\|_{2} \| \Delta u\|_2.
		%\end{split}
	\end{equation}
	For the third term of the right-hand-side of~\eqref{eq:first}, using the Cauchy-Schwarz inequality and Rellich inequality~\eqref{eq:Rellich} one has
	\begin{equation}\label{eq:rhsII}
		\Big| d\Re \int_{\R^d} f\bar{u} \Big|
		\leq d\||x|^2 f\|_{2} \left\| \frac{u}{~|x|^2}\right\|_2
		\leq \frac{4}{d-4} \||x|^2 f\|_{2}\|\Delta u\|_2.
	\end{equation}
	Plugging~\eqref{eq:rhsI} and~\eqref{eq:rhsII} in~\eqref{eq:first} and dividing by 4 we obtain
	\begin{equation*}
		\int_{\R^d} |\Delta u|^2 
		- \frac{1}{4}  \int_{\R^d} [x\cdot \nabla V]_+|u|^2	
		\leq\frac{2(d-2)}{d(d-4)}\||x|^2 f\|_{2} \|\Delta u \|_2.
	\end{equation*}
	Using the repulsivity condition~\eqref{eq:sa} for $V$ gives
	\begin{equation*}
		(1-a)\int_{\R^d} |\Delta u|^2
		\leq \frac{2(d-2)}{d(d-4)}\||x|^2 f\|_{2} \|\Delta u \|_2.
	\end{equation*}
	Dividing by $(1-a)\|\Delta u\|_2$ we get the thesis.
\end{proof}
Once we have Theorem~\ref{thm:apriori-sa}, the proof of Theorem~\ref{thm:resolvent-est-sa} follows as a direct application of the Rellich inequality~\eqref{eq:Rellich}. 
\begin{proof}[Proof of Theorem~\ref{thm:resolvent-est-sa}]
From the Rellich inequality~\eqref{eq:Rellich} and the a priori estimate~\eqref{eq:apriori}, one has that for any $u\in H^2(\R^d)$ solution of the resolvent equation~\eqref{eq:resolventV} the following chain of inequalities hold true
\begin{equation*}
	\|r^{-2}u\|_2\leq \frac{1}{\sqrt{C_R}}\|\Delta u\|_2\leq \frac{1}{\sqrt{C_R}}\frac{1}{1-a}\frac{2(d-2)}{d(d-4)}\|r^2f\|_2.
	\qedhere
\end{equation*}
\end{proof}

\section{Non-self-adjoint setting}\label{sec:non-self-adjoint}
This section is devoted to the proof of the results stated for non-self-adjoint operators, namely Theorems~\ref{thm:main-nsa}--\ref{thm:apriori}. In the proof of these we will need the following non-self-adjoint version of Lemma~\ref{lemma:harmonic-zero}
\begin{lemma}\label{lemma:harmonic-zero-nsa}
	Under the assumption of Theorem~\ref{thm:main-nsa}, given any solution $u\in H^2(\R^d)$ of~\eqref{eq:eigenv-eq}. If $\Delta u=0$ then $u=0.$
\end{lemma}
\begin{proof}
Again, we present a proof which does not require Liouville's theorem.
	As in the proof of Lemma~\ref{lemma:harmonic-zero} one has that since $\Delta u=0,$ in particular $u$ satisfies~\eqref{eq:deltau-zero}.
	From this we want to show that $u=0.$ We distinguish two cases:
	\begin{itemize}
		\item If $z \neq 0.$ We multiply~\eqref{eq:deltau-zero} by $\overline{u}$ and  we integrate over $\R^d$ obtaining
		\begin{equation*}
			\int_{\R^d} V |u|^2 = z \int_{\R^d} |u|^2.
		\end{equation*}
		In particular one has
		\begin{equation*}
			\begin{split}
				|z| \int_{\R^d} |u|^2&\leq \int_{\R^d} |V||u|^2\\
				&\leq \||x|^2 V u\|_{2}\Big\|\frac{u}{|x|^2}\Big\|_{2}\\
				&\leq \frac{a}{\sqrt{C_\textup{HR}} \sqrt{C_\textup{R}}} \|\Delta u\|_2^2\\
				&=0,
			\end{split}
		\end{equation*}
		where we used, in order, the Cauchy-Schwarz inequality, assumption~\eqref{eq:smallness} and the Hardy-Rellich and Rellich inequalities~\eqref{eq:HR} and~\eqref{eq:Rellich}, respectively. This gives
		\begin{equation*}
			|z| \int_{\R^d} |u|^2\leq 0,
		\end{equation*}
		that implies $u=0.$
	\item If $z=0.$ We proceed in the same way as in the proof of Lemma~\ref{lemma:harmonic-zero}. If $V=0,$ then $u$ solves $\Delta^2 u=0$ which implies $u=0$ as $\sigma_\textup{p}(\Delta^2)=\varnothing.$ So we can assume that $V\neq 0.$ Suppose by contradiction that $u\neq 0,$ thus $\psi:=|V|^{1/2} u$ is also non trivial. Indeed, since $u$ solves $\Delta^2 u + Vu=0,$ in particular one has $\Delta^2 u=- V |V|^{-1/2} |V|^{1/2} u.$ Thus if $\psi$ were trivial then $0\in \sigma_\textup{p}(\Delta^2)$ that is a contradiction.  Nevertheless, using the computations of the previous case one gets
	\begin{equation*}
		\int_{\R^d} |V||u|^2
		\leq \frac{a}{\sqrt{C_\textup{HR}} \sqrt{C_\textup{R}}} \|\Delta u\|_2^2
		=0.
	\end{equation*}
	Thus $\psi:=|V|^{1/2}u=0$ which is a contradiction. This concludes the proof.
	\qedhere
	\end{itemize} 
\end{proof}

\subsection{Absence of eigevalues outside a cone: Proof of Theorem~\ref{thm:absence-out}}
We start from the proof of Theorem~\ref{thm:absence-out} 
which excludes presence of eigenvalues outside a cone containing the positive semi-axis with varying opening.
This can be proved easily through the method of multipliers just using the identity involving the easiest symmetric multiplier.

\begin{proof}[Proof of Theorem~\ref{thm:absence-out}]
Assume by contradiction that $u$ is a solution of the eigenvalue problem~\eqref{eq:eigenv-eq} with $|\Im z|>\delta \Re z.$
Using identities~\eqref{eq:S1} and~\eqref{eq:S2} with $g=\pm 1$ and $f=-Vu$ gives 
\begin{equation}\label{eq:p}
	\pm\int_{\R^d} |\Delta u|^2 \mp \Re z \int_{\R^d} |u|^2=\pm \Re\int_{\R^d} f \bar{u},
\end{equation}
and 
\begin{equation}\label{eq:pm}
	\mp \Im z \int_{\R^d} |u|^2=\pm \Im \int_{\R^d} f\bar{u}.
\end{equation} 
Summing the first of~\eqref{eq:p} multiplied by $\delta$ to~\eqref{eq:pm} one has
\begin{equation*}
	\delta\int_{\R^d} |\Delta u|^2 -\delta\Re \int_{\R^d} f \bar{u} \mp \Im \int_{\R^d} f\bar{u} = (\delta \Re z\pm \Im z)\int_{\R^d} |u|^2.
\end{equation*}
The terms involving $f$ (thus $V$) can be estimated as follows
\begin{equation*}
	\left|-\delta\Re \int_{\R^d} f\bar{u} \mp\Im\int_{\R^d} f\bar{u}\right|
	\leq (\delta +1)\||x|^2 f\|_2 \left\|\frac{u}{|x|^2} \right\|_2
	\leq \frac{\delta +1}{\sqrt{C_R}} \||x|^2 f\|_2 \|\Delta u\|_2,
\end{equation*}
where in the last step we used the Rellich inequality~\eqref{eq:Rellich}. 

Recalling that we chose $f=-Vu$ and using condition~\eqref{eq:V-cond} we get
\begin{equation*}
\delta \left[1-\left(1+ \frac{1}{\delta} \right)\frac{a_\delta}{\sqrt{C_R}}\right]\int_{\R^d} |\Delta u|^2 \leq (\delta \Re z\pm \Im z)\int_{\R^d} |u|^2.	
\end{equation*}
This requires $\delta \Re z\pm \Im z\geq 0$ unless $u=0.$ Since $|\Im z|>\delta \Re z$ then $u=0.$ This concludes the proof.
\end{proof}

\subsection{Total absence of eigenvalues: Proof of Theorem~\ref{thm:main-nsa}}

Following the scheme described in Section~\ref{Sec.sketch},
in order to prove Theorem~\ref{thm:main-nsa} we will need a different strategy if the spectral parameter $z$ lies in the region $\mathcal{S}_{\textup{pos}}:=\{ z\in \C\colon \Re z\geq 0\}$ or if it lies in the region $\mathcal{S}_\textup{neg}:=\{z\in \C\colon \Re z<0\}.$ The case of $z\in \mathcal{S}_\textup{neg}$ is standard and can be proved essentially as in the proof of Theorem~\ref{thm:absence-out}. In the region $\mathcal{S}_{\textup{pos}}:=\{ z\in \C\colon \Re z\geq 0\}$ one needs the decomposition~\eqref{eq:decomposition-op} and reduces the analysis to the study of Schrödinger operators. 
More precisely, we will need the following two results about uniform resolvent estimates for solutions of the resolvent equation associated to the Laplacian. In the first result, Lemma~\ref{lemma:FKV} below, we shall consider the situation when the spectral parameter might approach or lie on the real line. The second result, Lemma~\ref{lemma:schroe} below, instead covers only the case of eigenvalues far from the essential spectrum.  
\begin{lemma}\label{lemma:FKV}
	Let $d\geq 3.$
	Assume $\kappa\in \C$ such that 
	$\Re \kappa\geq |\Im \kappa|$ 
	and $f\in L^2(\R^d, |x|dx).$ Then any solution $u\in H^1(\R^d)$ of the equation $(\Delta + \kappa)u=f$ satisfies
	\begin{equation}\label{eq:FKV}
		\|\nabla u^-\|_2^2
		\leq \frac{2d(d-3)}{d-2} \||x|f\|_2\|\nabla u^-\|_2 
		+ \frac{\sqrt{2}}{\sqrt{d-2}} \||x|f\|_2^{3/2}\|\nabla u^-\|_2^{1/2},
	\end{equation}
	where $u^-(x):=e^{-i(\Re \kappa)^{1/2}\sgn(\Im \kappa)|x|}u(x).$
\end{lemma}
\begin{proof}
	This result follows as a byproduct of the application
of the method of multipliers developed in~\cite{FanelliKrejcirikVega18-JST} (\emph{cfr.} identity (62) in~\cite{FanelliKrejcirikVega18-JST} and the estimates afterwards).
We refer to this work and particularly to~\cite{CossettiFanelliKrejcirik20} for the details of the proof.
\end{proof}

In the next lemma we show some suitable weighted resolvent estimates for the Schrödinger operator when we are outside a sector including the positive semi-axis. As expected, here no change of gauge are necessary.
\begin{lemma}\label{lemma:schroe}
Let $d\geq 5.$
Assume $\kappa\in \C$ such that $\Re \kappa \geq |\Im \kappa|$ and $f \in L^2(\R^d, |x|^2dx).$ Then any solution $v\in H^1(\R^d)$ of the equation $(\Delta-\kappa)v=f$ is such that $|x|^2 |\nabla v|^2 + |x|^2|v|^2\in L^1(\R^d)$ and one has
\begin{equation}\label{eq:uniform2}
\||x||\nabla v|\|_2\leq \frac{2}{d-4} \||x|^2 f\|_2.
\end{equation}
\end{lemma} 
\begin{proof}
The proof is obtained through the method of multipliers: As a starting point we want to approximate $v$ with compactly supported functions defining, for a given positive number $R,$ the auxiliary approximating family $v_R:=v \mu_R,$ with $\mu_R(x):=\mu(|x|R^{-1})$ and where $\mu$ was defined in~\eqref{eq:mucutoff}.

If $v\in H^1(\R^d)$ is any solution of $(\Delta - \kappa)v=f,$ then $v_R\in H^1(\R^d)$ and solves in a weak sense the following related problem
\begin{equation}\label{eq:appr-compact}
	(\Delta -\kappa)v_R=f_R + \text{err}(R),
\end{equation}
where 
\begin{equation}\label{eq:err(R)}
	\text{err}(R):=\Delta \mu_R v + 2\nabla \mu_R \nabla v.
\end{equation}
Multiplying the approximating equation~\eqref{eq:appr-compact} by $\phi v_R,$ for $\phi\colon \R^d \to \R$ to be chosen later but such that $\phi v_R\in H^1(\R^d)$ (since the support of $v_R$ is compact, any locally bounded $\phi$ with locally bounded partial derivatives is admissible), integrating over $\R^d,$ integrating by parts, taking the real part of the resulting identity and eventually integrating by parts again, one has
\begin{equation*}
	-\int_{\R^d}\phi |\nabla v_R|^2 
	+ \frac{1}{2}\int_{\R^d} \Delta \phi |v_R|^2 
	- \Re \kappa \int_{\R^d} \phi |v_R|^2
	=\Re \int_{\R^d} f_R\phi \overline{v_R}
	+ \Re \int_{\R^d} \text{err}(R) \phi \overline{v_R}.
\end{equation*}
Choosing $\phi=|x|^2$ and observing that $\Delta \phi=2d$ we obtain
\begin{equation}\label{eq:identity-Re}
	\int_{\R^d} |x|^2 |\nabla v_R|^2
	=d\int_{\R^d}|v_R|^2 
	- \Re \kappa \int_{\R^d} |x|^2 |v_R|^2
	-\Re \int_{\R^d} |x|^2 f_R \overline{v_R}
	- \Re \int_{\R^d} |x|^2 \text{err}(R) \overline{v_R}.
\end{equation}
Observe that since $\Re\kappa\geq |\Im\kappa|,$ in particular $\Re \kappa \geq 0,$ thus the term depending on the spectral parameter in~\eqref{eq:identity-Re} can be discarded obtaining
\begin{equation}\label{eq:ineq-Re-R}
	\int_{\R^d}|x|^2 |\nabla v_R|^2
	\leq d\int_{\R^d}|v_R|^2
	-\Re \int_{\R^d} |x|^2 f_R \overline{v_R}
	- \Re \int_{\R^d} |x|^2 \text{err}(R) \overline{v_R}.
\end{equation}
Passing to the limit $R$ to infinity one gets
\begin{equation}\label{eq:ineq-Re}
	\int_{\R^d}|x|^2 |\nabla v|^2
	\leq d\int_{\R^d}|v|^2
	-\Re \int_{\R^d} |x|^2 f \overline{v}.
\end{equation}

Indeed 
\begin{equation*}
	\int_{\R^d} |x|^2 |\nabla(v_R-v)|^2
	\leq 2 \int_{\R^d} |x|^2 |\nabla v|^2|\mu_R-1|^2
	+ \int_{\R^d} |x|^2 |v|^2 |\nabla \mu_R|^2.
\end{equation*}
The first term of the right-hand-side tends to zero thanks to the monotone convergence theorem. We see now the second term of the right-hand-side:
\begin{equation*}
	\int_{\R^d} |x|^2 |v|^2 |\nabla \mu_R|^2
	\leq c \int_{R\leq |x|\leq 2R} |v|^2.
\end{equation*}
Since $v\in L^2(\R^d),$ the right-hand-side tends to zero as $R$ goes to infinity. Using the previous estimates one gets
\begin{equation}\label{eq:limitLHS}
	\int_{\R^d} |x|^2|\nabla v_R|^2 \xrightarrow{R\to \infty} \int_{\R^d} |x|^2 |\nabla v|^2.
\end{equation}
Similarly, one has
\begin{equation}\label{eq:limit1RHS}
	\int_{\R^d} |v_R|^2 \xrightarrow{R\to \infty} \int_{\R^d} |v|^2.
\end{equation}
As for the term depending on $f$ we have
\begin{equation*}
	\begin{split}
	 \int_{\R^d} \Big||x|^2f_R \overline{v_R}- |x|^2 f\overline{v} 	\Big|
	 &\leq \int_{\R^d} |x|^2|f||\mu_R-1||v_R|
	 + \int_{\R^d} |x|^2|f||v||\mu_R-1|\\
	 &\leq \||x|^2 f (\mu_R-1)\|_2\|v_R\|_2 
	 + \||x|^2 f\|_2 \|v (\mu_R-1)\|_2.
	\end{split}
\end{equation*}
Since $|x|^2f\in L^2(\R^d),$ using the Lebesgue convergence theorem, from the previous estimates one gets
\begin{equation}\label{eq:limitfRHS}
	\int_{\R^d} |x|^2f_R \overline{v_R} \xrightarrow{R\to \infty} \int_{\R^d} |x|^2 f \overline{v}.
\end{equation}  
Now we consider the term depending on $\text{err}(R).$ We will show that this term goes to zero as $R$ goes to infinity. First of all, integrating by parts, one gets
\begin{equation*}
	\begin{split}
	\Re \int_{\R^d} |x|^2 \text{err}(R)\overline{v_R}
	&= \int_{\R^d} |x|^2 \Delta \mu_R  |v|^2 \mu_R
	+2\Re \int_{\R^d} |x|^2 \nabla \mu_R \nabla v \,\overline{v} \mu_R\\
	&=-2 \int_{\R^d} x \mu_R \nabla \mu_R |v|^2
	+ \int_{\R^d} |x|^2 |\nabla \mu_R|^2 |v|^2
	+2 \int_{\R^d} |x|^2 \mu_R \Delta \mu_R|v|^2.
	\end{split}
\end{equation*}
From this, one has the following inequality
\begin{equation*}
	\int_{\R^d} |x|^2|\text{err}(R)||v_R|
	\leq  c \int_{R\leq |x|\leq 2R} |v|^2,
\end{equation*}
for a suitable constant $c$ not dependent on $R.$
This gives
\begin{equation}\label{eq:limiteRHS}
\int_{\R^d} |x|^2 \text{err}(R) \overline{v_R} \xrightarrow{R\to \infty} 0.
\end{equation}
Passing to the limit in~\eqref{eq:ineq-Re-R} and using~\eqref{eq:limitLHS}--\eqref{eq:limiteRHS} gives~\eqref{eq:ineq-Re}.

Now we can come back to inequality~\eqref{eq:ineq-Re}.
To estimate the right-hand-side of~\eqref{eq:ineq-Re} we need the following weighted Hardy inequality
\begin{equation}\label{eq:weighted-Hardy}
	\int_{\R^d} |x|^{2\gamma} |\nabla \psi|^2\geq \frac{(d-2+2\gamma)^2}{4} \int_{\R^d} \frac{|\psi|^2}{|x|^{2-2\gamma}}, \qquad \gamma \in \R, \quad \psi\in C_0^\infty(\R^d\setminus \{0\}).
\end{equation}
Inequality~\eqref{eq:weighted-Hardy} can be proved in a similar way as the classical Hardy inequality in $d\geq 3$ (see, \emph{e.g.}~\cite[Prop.2.4]{CassanoPizzichillo18}). More precisely, we will need inequality~\eqref{eq:weighted-Hardy} for $\gamma=1,$ namely
\begin{equation}\label{eq:weighted-Hardy-gamma1}
\int_{\R^d} |x|^{2} |\nabla \psi|^2\geq \frac{d^2}{4} \int_{\R^d}|\psi|^2, \qquad \psi\in C_0^\infty(\R^d\setminus \{0\}).
\end{equation}
Using~\eqref{eq:weighted-Hardy-gamma1} in~\eqref{eq:ineq-Re} and the Cauchy-Schwarz inequality we obtain
\begin{equation*}
	\begin{split}
	\int_{\R^d}|x|^2 |\nabla v|^2&\leq d\int_{\R^d}|v|^2+ \||x|^2f\|_2\|v\|_2\\
	&\leq \frac{4}{d} \int_{\R^d} |x|^2 |\nabla v|^2 + \frac{2}{d} \||x|^2 f\|_2\||x||\nabla v|\|_2,
	\end{split}
\end{equation*}
which gives
\begin{equation*}
	\left(1-\frac{4}{d}\right)\int_{\R^d}|x|^2 |\nabla v|^2
	\leq \frac{2}{d} \||x|^2 f\|_2\||x||\nabla v|\|_2.
\end{equation*}
Notice that $1-\frac{4}{d}>0$ if $d\geq 5.$ Thus one finally gets
\begin{equation*}
	\||x||\nabla v|\|_2\leq \frac{2}{d-4} \||x|^2 f\|_2.
	\qedhere
\end{equation*} 
\end{proof}

The proof of our results 
in Theorems~\ref{thm:main-nsa}, \ref{thm:resolvent-est} 
and~\ref{thm:apriori} 
follows immediately from the following proposition for the free biharmonic operator.
\begin{proposition}\label{prop:apriori}
Let $d\geq 5.$ Given any $z\in \C$ and 
$f\in L^2(\R^d, |x|^2\,dx),$ any solution $u\in H^2(\R^d)$ of the equation $(\Delta^2-z)u=f$ satisfies
	\begin{itemize}
		\item for $z\in \mathcal{S}_\textup{pos},$ \emph{i.e.} $\Re z\geq 0,$ 
		\begin{equation}\label{eq:lr-bis}
			\|\nabla (\partial_j u)^-\|_2^2
		\leq \frac{4d(d-3)}{(d-2)(d-4)} \||x|^2f\|_2\|\nabla (\partial_j u)^-\|_2 
		+ \frac{4}{(d-4)\sqrt{(d-2)(d-4)}} \||x|^2f\|_2^{3/2}\|\nabla (\partial_j u)^-\|_2^{1/2},
		\end{equation}
		for $j\in \{1,2,\dots,d\},$
		\item for $z\in \mathcal{S}_\textup{neg},$ \emph{i.e.} $\Re z<0,$
		\begin{equation}\label{eq:ud-bis}
			\|\Delta u\|_2\leq C_R^{-1/2}\||x|^2 f\|_2,
		\end{equation}
		 with $C_R$ as in~\eqref{eq:Rellich}.
	\end{itemize}
	Here, for any suitable function $v,$ we denoted with $v^-$ the auxiliary function already defined in~\eqref{eq:auxiliary}.
\end{proposition}
\begin{proof}
The proof will be different depending on whether $z$ belongs to $\mathcal{S}_\textup{pos}$ or $\mathcal{S}_\textup{neg}.$ For this reason we treat the two cases separately.
\begin{description}[style=unboxed,leftmargin=0cm]
\item[Region $\mathcal{S}_\textup{pos}$.]
As already mentioned we rewrite the resolvent equation as 
\begin{equation}\label{eq:decomposition}
	(\Delta-\sqrt{z})(\Delta+\sqrt{z})u=f
\end{equation} 
and we introduce the following notation 
\begin{equation}\label{eq:v-def}
	(\Delta + \sqrt{z})u=:v.
\end{equation}
In particular, this means that for any $j\in \{1,2,\dots, d\}$ then $\partial_j u$ is a solution of the following resolvent equation associated to the Laplacian
\begin{equation}\label{eq:resolvent-Schr1}
	(\Delta + \sqrt{z})\partial_j u=\partial_j v.
\end{equation}
Notice that if $z\in \mathcal{S}_\textup{pos}$ then in particular one has $\Re \sqrt{z} \geq |\Im \sqrt{z}|$ (see also Figure~\ref{fig:sectors}). Thus, one can use estimate~\eqref{eq:FKV} in Lemma~\ref{lemma:FKV} to obtain that for $(\partial_j u)^-$
%defined as
%\begin{equation*}
%	(\partial_j u)^-(x):=e^{-i(\Re k)^{1/2}\sgn(\Im k) |x|}\, \partial_j u(x), 
%	\qquad k:=\sqrt{z},
%\end{equation*}  
the \emph{uniform} estimate
\begin{equation}\label{eq:uniform1}
	\|\nabla (\partial_j u)^-\|_2^2
		\leq \frac{2d(d-3)}{d-2} \||x|\partial_j v\|_2\|\nabla (\partial_j u)^-\|_2 
		+ \frac{\sqrt{2}}{\sqrt{d-2}} \||x|\partial_j v\|_2^{3/2}\|\nabla (\partial_j u)^-\|_2^{1/2}
\end{equation}
holds true, with $\partial_j u,$ $j\in \{1,2,\dots, d\}$ being the solution of~\eqref{eq:resolvent-Schr1}. 
Since $v$ satisfies $(\Delta -\sqrt{z})v=f$
(see~\eqref{eq:decomposition} and~\eqref{eq:v-def}) we can use Lemma~\ref{lemma:schroe}: plugging~\eqref{eq:uniform2} in~\eqref{eq:uniform1} we get~\eqref{eq:lr-bis}. This concludes the analysis in the region $\mathcal{S}_\textup{pos}.$
\item[Region $\mathcal{S}_\textup{neg}$.] We start noticing that since $z\in \mathcal{S}_\textup{neg}$ then $\Re z< 0.$ Thus we can use a similar reasoning as in the proof of Theorem~\ref{thm:absence-out} to get~\eqref{eq:ud-bis}:
from identity~\eqref{eq:S1} with $g=1$ we have 
\begin{equation*}
	\int_{\R^d} |\Delta u|^2 - \Re z \int_{\R^d} |u|^2= \Re\int_{\R^d} f \bar{u}.
\end{equation*}
 Discarding the positive term involving $\Re z$ and using the Cauchy-Schwarz inequality one has
 \begin{equation*}
\int_{\R^d} |\Delta u|^2 \leq \||x|^2f\|_2\|\frac{u}{|x|^2}\|_2\leq  \frac{1}{\sqrt{C_R}}\||x|^2f\|_2\|\Delta u\|_2,
 \end{equation*}
 where in the last inequality we have used the Rellich inequality~\eqref{eq:Rellich}. Dividing by $\|\Delta u\|_2$ we obtain~\eqref{eq:ud-bis}.
 \qedhere
\end{description}
\end{proof}

Using Proposition~\ref{prop:apriori} the proofs of 
Theorems~\ref{thm:main-nsa}, \ref{thm:resolvent-est} and~\ref{thm:apriori} 
follow easily once the following lemma is proved.
\begin{lemma}\label{lemma:f}
	Let $d\geq 5.$ 
	Under the hypotheses of Theorem~\ref{thm:main-nsa},
	for any $\psi\in H^2(\R^d)$ one has
	\begin{align}
	\label{eq:V1}
		\||x|^2V\psi\|_2&\leq \frac{a}{\sqrt{C_\textup{H}}} \sum_{l=1}^d \|\nabla (\partial_l \psi)^-\|_2,\\
		\label{eq:V2}
		\||x|^2V\psi\|_2&\leq \frac{a}{\sqrt{C_\textup{HR}}} \|\Delta \psi\|_2,
	\end{align}
	where $a$ is as in~\eqref{eq:smallness}, $C_\textup{H}$ and $C_\textup{HR}$ are the Hardy constant (see~\eqref{eq:Hardy}) and the Hardy-Rellich constant (see~\eqref{eq:HR}), respectively and where $(\partial_l \psi)^-$ is defined as in~\eqref{eq:auxiliary}.
\end{lemma}
\begin{proof}
	The proof is not difficult 
	but we provide it here for sake of completeness. We start with showing~\eqref{eq:V1}. Using~\eqref{eq:smallness}, the fact that $|(\partial_j u)^-|=|\partial_j u|$ and the Hardy inequality~\eqref{eq:Hardy} one has
	\begin{equation*}
		\||x|^2 V\psi\|_2^2
		\leq a^2 \left \|\frac{\nabla \psi}{|x|^2} \right\|_2^2
		=a^2 \sum _{l=1}^d \left\| \frac{|(\partial_l \psi)^-|}{|x|^2}\right\|_2^2
		\leq \frac{a^2}{C_\textup{H}} \sum_{l=1}^d \|\nabla (\partial_l \psi)^-\|_2^2.
	\end{equation*} 
	As for~\eqref{eq:V2}, simply using~\eqref{eq:smallness} and the Hardy-rellich inequality~\eqref{eq:HR} we get
	\begin{equation*}
		\||x|^2 V\psi\|_2^2
		\leq a^2 \left \|\frac{\nabla \psi}{|x|^2} \right\|_2^2
		\leq \frac{a^2}{C_\textup{HR}} \|\Delta \psi\|_2^2.
		\qedhere
	\end{equation*} 
\end{proof}

Now we are in position to prove our result on the absence of eigenvalues, namely Theorem~\ref{thm:main-nsa}.
\begin{proof}[Proof of Theorem~\ref{thm:main-nsa}]
We assume by contradiction that there exists a non trivial 
(weak) solution $u\in H^2(\R^d)$ of the eigenvalue equation
\begin{equation}\label{eq:evs-eq-f}
	\Delta^2 u -zu=f,
	\qquad f:=-Vu.
\end{equation}
We consider separately the case when $z\in \mathcal{S}_\textup{pos},$ \emph{i.e.} $\Re z\geq 0$ and the case $z\in \mathcal{S}_\textup{neg},$ \emph{i.e.} $\Re z <0.$
If $z\in \mathcal{S}_\textup{pos}$ we use estimate~\eqref{eq:lr-bis} from Proposition~\ref{prop:apriori} and~\eqref{eq:V1} to have
\begin{equation*}
	\begin{split}
	\sum_{j=1}^d \|\nabla (\partial_j u)^-\|_2^2
	&\leq \frac{4d(d-3)}{(d-2)(d-4)} \frac{a}{\sqrt{C_\textup{H}}} \Big( \sum_{j=1}^d \|\nabla (\partial_j u)^-\|_2 \Big)^2\\
		&\qquad \qquad + \frac{4}{(d-4)\sqrt{(d-2)(d-4)}}\frac{a^{3/2}}{\sqrt{C_\textup{H}}^{3/2}} \Big( \sum_{l=1}^d \|\nabla (\partial_l u)^-\|_2 \Big)^{3/2} \sum_{j=1}^d \|\nabla (\partial_j u)^-\|_2^{1/2}\\
		&\leq \left[\frac{4d^2(d-3)}{(d-2)(d-4)} \frac{a}{\sqrt{C_\textup{H}}} + \frac{4d \sqrt{d}}{(d-4)\sqrt{(d-2)(d-4)}}\frac{a^{3/2}}{\sqrt{C_\textup{H}}^{3/2}} \right]\sum_{j=1}^d \|\nabla (\partial_j u)^-\|_2^2.
		\end{split}
\end{equation*}
Since $a$ satisfies~\eqref{eq:smallness-a} we get $u=0.$ Thus, there are no eigenvalues $z\in \mathcal{S}_\textup{pos}.$ Let us consider now $z\in \mathcal{S}_\textup{neg}.$ From~\eqref{eq:ud-bis} and~\eqref{eq:V2} we have 
\begin{equation*}
\|\Delta u\|_2 \leq \frac{a}{\sqrt{C_\textup{R}}\sqrt{C_\textup{HR}}} \|\Delta u\|_2.
\end{equation*}
Since $a$ satisfies~\eqref{eq:smallness-a}, in particular $a<\sqrt{C_\textup{R}}\sqrt{C_\textup{HR}},$ thus, as above, one concludes that $u=0.$
\end{proof}

\subsection{A priori and resolvent estimates: 
Proof of Theorems~\ref{thm:resolvent-est} and~\ref{thm:apriori}}
As already mentioned, Theorem~\ref{thm:resolvent-est} follows easily once Theorem~\ref{thm:apriori} is available. Thus we prove the a priori estimates stated in Theorem~\ref{thm:apriori} first.
\begin{proof}[Proof of Theorem~\ref{thm:apriori}]
	We start considering the case $z\in \mathcal{S}_\textup{pos}.$ Using estimate~\eqref{eq:lr-bis} in Proposition~\ref{prop:apriori} with $f$ replaced by $-Vu + f$ one has
	\begin{multline}\label{eq:apriori-final}
		(1-c_{a,d}) \sum_{j=1}^d \|\nabla(\partial_j u)^-\|_2^2\\
		\leq 
		\frac{4d(d-3)}{(d-2)(d-4)} \||x|^2f\|_2 \sum_{j=1}^d\|\nabla (\partial_j u)^-\|_2 
		+ \frac{4}{(d-4)\sqrt{(d-2)(d-4)}} \||x|^2f\|_2^{3/2}\sum_{j=1}^d\|\nabla (\partial_j u)^-\|_2^{1/2},
	\end{multline} 
	where the terms involving $V$ are estimated as in the proof of Theorem~\ref{thm:main-nsa} above, and the constant $c_{a,d}$ is taken from that proof, namely
	\begin{equation*}
		c_{a,d}
		=\frac{4d^2(d-3)}{(d-2)(d-4)} \frac{a}{\sqrt{C_\textup{H}}} + \frac{4d \sqrt{d}}{(d-4)\sqrt{(d-2)(d-4)}}\frac{a^{3/2}}{\sqrt{C_\textup{H}}^{3/2}},
	\end{equation*}
	with $a$ as in~\eqref{eq:smallness-a} and $C_\textup{H}$ is the Hardy constant (\emph{cfr.}~\eqref{eq:Hardy}). Given $\varepsilon, \delta>0$ to be chosen later, Young's inequality gives
	\begin{equation*}
		\||x|^2f\|_2 \sum_{j=1}^d\|\nabla (\partial_j u)^-\|_2 
		\leq \frac{\||x|^2f\|_2^2}{2\varepsilon^2} + \varepsilon^2 d \sum_{j=1}^d \|\nabla(\partial_ju)^-\|_2^2,
	\end{equation*}	
	and 
	\begin{equation*}
		\||x|^2f\|_2^{3/2} \sum_{j=1}^d\|\nabla (\partial_j u)^-\|_2^{1/2} 
		\leq \frac{3}{4} \frac{\||x|^2f\|_2^2}{\delta^{4/3}} + \frac{\delta^4}{4} d^2 \sum_{j=1}^d \|\nabla(\partial_ju)^-\|_2^2.
	\end{equation*}	
	Using the latter in~\eqref{eq:apriori-final} gives
	\begin{multline*}
		\left(1-c_{a,d}
		-\frac{4d(d-3)}{(d-2)(d-4)} \varepsilon^2 d
		-\frac{4}{(d-4)\sqrt{(d-2)(d-4)}} \frac{\delta^4}{4}d^2 
		\right)\sum_{j=1}^d \|\nabla(\partial_ju)^-\|_2^2\\
		\leq \left(\frac{4d(d-3)}{(d-2)(d-4)} \frac{1}{2\varepsilon^2}
		+ \frac{4}{(d-4)\sqrt{(d-2)(d-4)}} \frac{3}{4\delta^{4/3}} \right)\||x|^2f\|_2^2.
	\end{multline*}	
	If $\varepsilon$ and $\delta$ are small enough then one has the a priori estimate~\eqref{eq:lr}.
	
	We consider now the case $z\in \mathcal{S}_\textup{neg}.$ Using estimate~\eqref{eq:ud-bis} in Proposition~\ref{prop:apriori} with $f$ replaced by $-Vu + f$ and estimate~\eqref{eq:V2} one has
	\begin{equation*}
		\left(1-\frac{a}{\sqrt{C_\textup{R}}\sqrt{C_\textup{HR}}} \right) \|\Delta u\|_2\leq \frac{\||x|^2f\|_2}{\sqrt{C_\textup{R}}},
	\end{equation*}
	where $a$ is as in~\eqref{eq:smallness-a} and $C_\textup{R}$ and $C_\textup{HR}$ are the Rellich and Hardy-Rellich constants, respectively (\emph{cfr.}~\eqref{eq:Rellich} and~\eqref{eq:HR}). The last estimate gives immediately~\eqref{eq:ud}.
\end{proof}

Now we can prove the resolvent estimate contained in Theorem~\ref{thm:resolvent-est}.
\begin{proof}[Proof of Theorem~\ref{thm:resolvent-est}]
First of all, observe that from the Hardy inequality~\eqref{eq:Hardy} and the weighted Hardy inequality~\eqref{eq:weighted-Hardy} with $\gamma=-1$ one has
\begin{equation*}
	\sum_{j=1}^d \|\nabla(\partial_j u)^-\|_2
	\geq \sqrt{C_\textup{H}} \sum_{j=1}^d \left\|\frac{(\partial_j u)^-}{|x|}\right\|_2
	\geq \sqrt{C_\textup{H}}  \left\|\frac{\nabla u}{|x|}\right\|_2
	\geq \sqrt{C_\textup{H}} \frac{(d-4)}{2} \left \|\frac{u}{|x|^2} \right\|_2,
\end{equation*}
here we have also used that $|(\partial_j u)^-|=|\partial_j u|.$
If $z\in \mathcal{S}_\textup{pos},$ namely if $ \Re z\geq 0$ then from~\eqref{eq:lr} and the latter, there exists a constant $c>0$ such that  
\begin{equation*}
	\left \|\frac{u}{|x|^2} \right\|_2
	\leq c \||x|^2 f\|_2.
\end{equation*}
If $z\in \mathcal{S}_\textup{neg},$ namely if $z\in \Re z<0,$ from~\eqref{eq:ud} and the Rellich inequality~\eqref{eq:Rellich} one has
\begin{equation*}
	\sqrt{C_\textup{R}}\left \|\frac{u}{|x|^2} \right\|_2
	\leq \|\Delta u\|_2
	\leq c_{a,d} \||x|^2f\|_2,
\end{equation*}
with $c_{a,d}$ as in~\eqref{eq:cad}.
This concludes the proof.
\end{proof}

\section*{Acknowledgements}

L.C was supported by the grant Ramón y Cajal RYC2021-032803-I  funded by MCIN/AEI/10.13039/50110
\noindent
0011033 and by the European Union NextGenerationEU/PRTR, the Deutsche Forschungsgemeinschaft (DFG, German
Research Foundation) -- Project-ID 258734477 -- SFB 1173 and by Ikerbasque.

L.F. was supported by project PID2021-123034NB-I00/AEI/10.13039/501100011033 funded by the Agencia Estatal de Investigaci\'on (Spain), by the project IT1615-22 funded by the Basque Government, and by Ikerbasque.

D.K. was supported by the EXPRO grant No. 20-17749X
of the Czech Science Foundation.

\bibliography{spectralref04}

\begin{thebibliography}{10}

\bibitem{AbramovAslanyanDavies01}
A.~A. Abramov, A.~Aslanyan, and E.~B. Davies.
\newblock Bounds on complex eigenvalues and resonances.
\newblock {\em J. Phys. A}, 34(1):57--72, 2001.

\bibitem{Amrein1996}
W.~O. Amrein, A.~B. de~Monvel, and V.~Georgescu.
\newblock {\em Some Examples of C0-Groups}.
\newblock Springer Basel, Basel, 1996.

\bibitem{BarceloFanelliRuizVilela2013}
J.~A. Barcel\'{o}, L.~Fanelli, A.~Ruiz, and M.~Vilela.
\newblock {\it {A} priori} estimates for the {H}elmholtz equation with
  electromagnetic potentials in exterior domains.
\newblock {\em Proc. Roy. Soc. Edinburgh Sect. A}, 143(1):1--19, 2013.

\bibitem{BarceloVegaZubeldia2013}
J.~A. Barceló, L.~Vega, and M.~Zubeldia.
\newblock The forward problem for the electromagnetic helmholtz equation with
  critical singularities.
\newblock {\em Advances in Mathematics}, 240:636--671, 2013.

\bibitem{BogliCuenin22}
S.~B{\"o}gli and J.-C. Cuenin.
\newblock Counterexample to the {L}aptev--{S}afronov conjecture.
\newblock {\em Communications in Mathematical Physics}, 2022.

\bibitem{BoussaidDAnconaFanelli11}
N.~Boussaid, P.~D'Ancona, and L.~Fanelli.
\newblock Virial identity and weak dispersion for the magnetic {D}irac
  equation.
\newblock {\em J. Math. Pures Appl. (9)}, 95(2):137--150, 2011.

\bibitem{BPST2004}
N.~Burq, F.~Planchon, J.~G. Stalker, and A.~S. Tahvildar-Zadeh.
\newblock Strichartz estimates for the wave and {S}chr\"{o}dinger equations
  with potentials of critical decay.
\newblock {\em Indiana Univ. Math. J.}, 53(6):1665--1680, 2004.

\bibitem{Cacciafesta11}
F.~Cacciafesta.
\newblock Virial identity and dispersive estimates for the {$n$}-dimensional
  {D}irac equation.
\newblock {\em J. Math. Sci. Univ. Tokyo}, 18(4):441--463 (2012), 2011.

\bibitem{CacciafestaDAnconaLuca16}
F.~Cacciafesta, P.~D'Ancona, and R.~Luc\`a.
\newblock Helmholtz and dispersive equations with variable coefficients on
  exterior domains.
\newblock {\em SIAM J. Math. Anal.}, 48(3):1798--1832, 2016.

\bibitem{CassanoCossettiFanelli21-lame}
B.~Cassano, L.~Cossetti, and L.~Fanelli.
\newblock Eigenvalue bounds and spectral stability of lam{\'e} operators with
  complex potentials.
\newblock {\em arXiv preprint arXiv:2101.09350}, 2021.

\bibitem{CassanoCossettiFanelli21}
B.~Cassano, L.~Cossetti, and L.~Fanelli.
\newblock Spectral enclosures for the damped elastic wave equation.
\newblock {\em arXiv preprint arXiv:2108.07676}, 2021.

\bibitem{CassanoDAncona16}
B.~Cassano and P.~D'Ancona.
\newblock Scattering in the energy space for the {NLS} with variable
  coefficients.
\newblock {\em Math. Ann.}, 366(1-2):479--543, 2016.

\bibitem{CassanoIbrogimovKrejcirikStampach20}
B.~Cassano, O.~O. Ibrogimov, D.~Krej\v{c}i\v{r}\'{\i}k, and F.~{\v{S}}tampach.
\newblock Location of eigenvalues of non-self-adjoint discrete {D}irac
  operators.
\newblock {\em Ann. Henri Poincar\'{e}}, 21(7):2193--2217, 2020.

\bibitem{CassanoPizzichillo18}
B.~Cassano and F.~Pizzichillo.
\newblock Self-adjoint extensions for the {D}irac operator with {C}oulomb-type
  spherically symmetric potentials.
\newblock {\em Lett. Math. Phys.}, 108(12):2635--2667, 2018.

\bibitem{CassanoPizzichilloVega20}
B.~Cassano, F.~Pizzichillo, and L.~Vega.
\newblock A {H}ardy-type inequality and some spectral characterizations for the
  {D}irac-{C}oulomb operator.
\newblock {\em Rev. Mat. Complut.}, 33(1):1--18, 2020.

\bibitem{Cossetti17}
L.~Cossetti.
\newblock Uniform resolvent estimates and absence of eigenvalues for {L}am\'{e}
  operators with complex potentials.
\newblock {\em J. Math. Anal. Appl.}, 455(1):336--360, 2017.

\bibitem{Cossetti22}
L.~Cossetti.
\newblock Bounds on eigenvalues of perturbed {L}am\'{e} operators with complex
  potentials.
\newblock {\em Math. Eng.}, 4(5):Paper No. 037, 29, 2022.

\bibitem{CossettiFanelliKrejcirik20}
L.~Cossetti, L.~Fanelli, and D.~Krej\v{c}i\v{r}\'{\i}k.
\newblock Absence of eigenvalues of {D}irac and {P}auli {H}amiltonians via the
  method of multipliers.
\newblock {\em Comm. Math. Phys.}, 379(2):633--691, 2020.

\bibitem{CossettiKrejcirik20}
L.~Cossetti and D.~Krej\v{c}i\v{r}\'{\i}k.
\newblock Absence of eigenvalues of non-self-adjoint {R}obin {L}aplacians on
  the half-space.
\newblock {\em Proc. Lond. Math. Soc.}, 121(3):584--616, 2020.

\bibitem{Cuenin17}
J.-C. Cuenin.
\newblock Eigenvalue bounds for {D}irac and fractional {S}chr\"{o}dinger
  operators with complex potentials.
\newblock {\em J. Funct. Anal.}, 272(7):2987--3018, 2017.

\bibitem{Cuenin19}
J.-C. Cuenin.
\newblock Eigenvalue estimates for bilayer graphene.
\newblock {\em Ann. Henri Poincar{\'e}}, 20(5):1501--1516, 2019.

\bibitem{CueninLaptevTretter14}
J.-C. Cuenin, A.~Laptev, and C.~Tretter.
\newblock Eigenvalue estimates for non-selfadjoint {D}irac operators on the
  real line.
\newblock {\em Ann. Henri Poincar\'{e}}, 15(4):707--736, 2014.

\bibitem{DAnconaFanelliSchiavone21}
P.~D’Ancona, L.~Fanelli, and N.~M. Schiavone.
\newblock Eigenvalue bounds for non-selfadjoint {D}irac operators.
\newblock {\em Math. Ann.}, pages 1--24, 2021.

\bibitem{Enblom16}
A.~Enblom.
\newblock Estimates for eigenvalues of {S}chr\"{o}dinger operators with
  complex-valued potentials.
\newblock {\em Lett. Math. Phys.}, 106(2):197--220, 2016.

\bibitem{ErdoganGoldbergGreen2023-preprint}
M.~B. Erdogan, M.~Goldberg, and W.~R. Green.
\newblock Dispersive estimates for higher order schr\"odinger operators with
  scaling-critical potentials, 2023.

\bibitem{ErdoganGoldbergGreen2023}
M.~B. Erdo\u{g}an, M.~Goldberg, and W.~R. Green.
\newblock Counterexamples to {$L^p$} boundedness of wave operators for
  classical and higher order {S}chr\"{o}dinger operators.
\newblock {\em J. Funct. Anal.}, 285(5):Paper No. 110008, 18, 2023.

\bibitem{ErdoganGreen2022}
M.~B. Erdo\u{g}an and W.~R. Green.
\newblock The {$L^p$}-continuity of wave operators for higher order
  {S}chr\"{o}dinger operators.
\newblock {\em Adv. Math.}, 404:Paper No. 108450, 41, 2022.

\bibitem{ErdoganGreen2023}
M.~B. Erdo\u{g}an and W.~R. Green.
\newblock A note on endpoint {$L^p$}-continuity of wave operators for classical
  and higher order {S}chr\"{o}dinger operators.
\newblock {\em J. Differential Equations}, 355:144--161, 2023.

\bibitem{ErdoganGreenToprak2021}
M.~B. Erdo\u{g}an, W.~R. Green, and E.~Toprak.
\newblock On the fourth order {S}chr\"{o}dinger equation in three dimensions:
  dispersive estimates and zero energy resonances.
\newblock {\em J. Differential Equations}, 271:152--185, 2021.

\bibitem{Eidus62}
D.~M. \`E\u{\i}dus.
\newblock On the principle of limiting absorption.
\newblock {\em Mat. Sb. (N.S.)}, 57 (99):13--44, 1962.

\bibitem{Evans-book}
W.~D. Evans.
\newblock {\em Partial differential equations}.
\newblock Graduate Studies in Mathematics, vol. 19. American Mathematical
  Society, Providence, RI, 1998.

\bibitem{Fanelli09}
L.~Fanelli.
\newblock Non-trapping magnetic fields and {M}orrey-{C}ampanato estimates for
  {S}chr\"{o}dinger operators.
\newblock {\em J. Math. Anal. Appl.}, 357(1):1--14, 2009.

\bibitem{FanelliKrejcirik19}
L.~Fanelli and D.~Krej\v{c}i\v{r}\'{\i}k.
\newblock Location of eigenvalues of three-dimensional non-self-adjoint {D}irac
  operators.
\newblock {\em Lett. Math. Phys.}, 109(7):1473--1485, 2019.

\bibitem{FanelliKrejcirikVega18-JFA}
L.~Fanelli, D.~Krej\v{c}i\v{r}\'{\i}k, and L.~Vega.
\newblock Absence of eigenvalues of two-dimensional magnetic {S}chr\"{o}dinger
  operators.
\newblock {\em J. Funct. Anal.}, 275(9):2453--2472, 2018.

\bibitem{FanelliKrejcirikVega18-JST}
L.~Fanelli, D.~Krej\v{c}i\v{r}\'{\i}k, and L.~Vega.
\newblock Spectral stability of {S}chr\"{o}dinger operators with subordinated
  complex potentials.
\newblock {\em J. Spectr. Theory}, 8(2):575--604, 2018.

\bibitem{FanelliVega2009}
L.~Fanelli and L.~Vega.
\newblock Magnetic virial identities, weak dispersion and {S}trichartz
  inequalities.
\newblock {\em Math. Ann.}, 344(2):249--278, 2009.

\bibitem{FengSofferWuYao2020}
H.~Feng, A.~Soffer, Z.~Wu, and X.~Yao.
\newblock Decay estimates for higher-order elliptic operators.
\newblock {\em Trans. Amer. Math. Soc.}, 373(4):2805--2859, 2020.

\bibitem{Feng-Soffer-Yao_2018}
H.~Feng, A.~Soffer, and X.~Yao.
\newblock Decay estimates and {S}trichartz estimates of fourth-order
  {S}chr{\"o}dinger operator.
\newblock {\em J. Funct. Anal.}, 274:605--658, 2018.

\bibitem{Frank11}
R.~L. Frank.
\newblock Eigenvalue bounds for {S}chr\"{o}dinger operators with complex
  potentials.
\newblock {\em Bull. Lond. Math. Soc.}, 43(4):745--750, 2011.

\bibitem{Frank18}
R.~L. Frank.
\newblock Eigenvalue bounds for {S}chr\"{o}dinger operators with complex
  potentials. {III}.
\newblock {\em Trans. Amer. Math. Soc.}, 370(1):219--240, 2018.

\bibitem{FrankLaptevLiebSeiringer06}
R.~L. Frank, A.~Laptev, E.~H. Lieb, and R.~Seiringer.
\newblock Lieb-{T}hirring inequalities for {S}chr\"{o}dinger operators with
  complex-valued potentials.
\newblock {\em Lett. Math. Phys.}, 77(3):309--316, 2006.

\bibitem{FrankSimon17}
R.~L. Frank and B.~Simon.
\newblock Eigenvalue bounds for {S}chr\"{o}dinger operators with complex
  potentials. {II}.
\newblock {\em J. Spectr. Theory}, 7(3):633--658, 2017.

\bibitem{GreenToprak2019}
W.~R. Green and E.~Toprak.
\newblock On the fourth order {S}chr\"{o}dinger equation in four dimensions:
  dispersive estimates and zero energy resonances.
\newblock {\em J. Differential Equations}, 267(3):1899--1954, 2019.

\bibitem{IbrogimovKrejcirikLaptev2021}
O.~O. Ibrogimov, D.~Krej\v{c}i\v{r}\'{\i}k, and A.~Laptev.
\newblock Sharp bounds for eigenvalues of biharmonic operators with complex
  potentials in low dimensions.
\newblock {\em Math. Nachr.}, 294(7):1333--1349, 2021.

\bibitem{IbrogimovStampach19}
O.~O. Ibrogimov and F.~{\v{S}}tampach.
\newblock Spectral enclosures for non-self-adjoint discrete {S}chr{\"o}dinger
  operators.
\newblock {\em Integral Equations and Operator Theory}, 91(6):1--15, 2019.

\bibitem{IkebeSaito1972}
T.~Ikebe and Y.~Saito.
\newblock Limiting absorption method and absolute continuity for the
  {S}chr\"{o}dinger operator.
\newblock {\em J. Math. Kyoto Univ.}, 12:513--542, 1972.

\bibitem{KrejcirikKurimaiova20}
D.~Krej{\v{c}}i{\v{r}}{\'\i}k and T.~Kurimaiov{\'a}.
\newblock From {L}ieb--{T}hirring inequalities to spectral enclosures for the
  damped wave equation.
\newblock {\em Integral Equations and Operator Theory}, 92(6):1--12, 2020.

\bibitem{LaptevSafronov09}
A.~Laptev and O.~Safronov.
\newblock Eigenvalue estimates for {S}chr\"{o}dinger operators with complex
  potentials.
\newblock {\em Comm. Math. Phys.}, 292(1):29--54, 2009.

\bibitem{LeeSeo19}
Y.~Lee and I.~Seo.
\newblock A note on eigenvalue bounds for {S}chr\"{o}dinger operators.
\newblock {\em J. Math. Anal. Appl.}, 470(1):340--347, 2019.

\bibitem{LiSofferYao2023}
P.~Li, A.~Soffer, and X.~Yao.
\newblock Decay estimates for fourth-order {S}chr\"{o}dinger operators in
  dimension two.
\newblock {\em J. Funct. Anal.}, 284(6):Paper No. 109816, 83, 2023.

\bibitem{Morawetz1968}
C.~S. Morawetz.
\newblock Time decay for the nonlinear {K}lein-{G}ordon equations.
\newblock {\em Proc. Roy. Soc. London Ser. A}, 306:291--296, 1968.

\bibitem{Mourre80}
E.~Mourre.
\newblock Absence of singular continuous spectrum for certain selfadjoint
  operators.
\newblock {\em Comm. Math. Phys.}, 78(3):391--408, 1980/81.

\bibitem{PerthameVega1999}
B.~Perthame and L.~Vega.
\newblock Morrey-{C}ampanato estimates for {H}elmholtz equations.
\newblock {\em J. Funct. Anal.}, 164(2):340--355, 1999.

\bibitem{PerthameVega2008}
B.~Perthame and L.~Vega.
\newblock Energy concentration and {S}ommerfeld condition for {H}elmholtz
  equation with variable index at infinity.
\newblock {\em Geom. Funct. Anal.}, 17(5):1685--1707, 2008.

\bibitem{ReedSimonIV}
M.~Reed and B.~Simon.
\newblock {\em Methods of modern mathematical physics. {IV}. {A}nalysis of
  operators}.
\newblock Academic Press [Harcourt Brace Jovanovich, Publishers], New
  York-London, 1978.

\bibitem{Rellich1943}
F.~Rellich.
\newblock \"{U}ber das asymptotische {V}erhalten der {L}\"{o}sungen von
  {$\Delta u+\lambda u=0$} in unendlichen {G}ebieten.
\newblock {\em Jber. Deutsch. Math.-Verein.}, 53:57--65, 1943.

\bibitem{Safronov10}
O.~Safronov.
\newblock Estimates for eigenvalues of the {S}chr\"{o}dinger operator with a
  complex potential.
\newblock {\em Bull. Lond. Math. Soc.}, 42(3):452--456, 2010.

\bibitem{SafronovLaptevFerrulli19}
O.~Safronov, A.~Laptev, and F.~Ferrulli.
\newblock Eigenvalues of the bilayer graphene operator with a complex valued
  potential.
\newblock {\em Anal. Math. Phys.}, 9(3):1535--1546, 2019.

\bibitem{Schlag2021}
W.~Schlag.
\newblock On pointwise decay of waves.
\newblock {\em J. Math. Phys.}, 62(6):Paper No. 061509, 27, 2021.

\bibitem{Schot1992}
S.~H. Schot.
\newblock Eighty years of {S}ommerfeld's radiation condition.
\newblock {\em Historia Math.}, 19(4):385--401, 1992.

\bibitem{Sommerfeld1912}
A.~Sommerfeld.
\newblock Die greensche funktion der schwingungslgleichung.
\newblock {\em Jahresbericht der Deutschen Mathematiker-Vereinigung},
  21:309--352, 1912.

\bibitem{TertikasZographopoulos2007}
A.~Tertikas and N.~B. Zographopoulos.
\newblock Best constants in the {H}ardy-{R}ellich inequalities and related
  improvements.
\newblock {\em Adv. Math.}, 209(2):407--459, 2007.

\bibitem{Weidmann1967}
J.~Weidmann.
\newblock {The virial theorem and its application to the spectral theory of
  Schrödinger operators}.
\newblock {\em Bulletin of the American Mathematical Society}, 73(3):452 --
  456, 1967.

\bibitem{Zubeldia2012}
M.~Zubeldia.
\newblock Energy concentration and explicit {S}ommerfeld radiation condition
  for the electromagnetic {H}elmholtz equation.
\newblock {\em J. Funct. Anal.}, 263(9):2832--2862, 2012.

\bibitem{Zubeldia2014}
M.~Zubeldia.
\newblock Limiting absorption principle for the electromagnetic {H}elmholtz
  equation with singular potentials.
\newblock {\em Proc. Roy. Soc. Edinburgh Sect. A}, 144(4):857--890, 2014.

\end{thebibliography}
\bibliographystyle{abbrv} %abbrv %alpha %apalike %plain

\end{document}